\documentclass[a4paper,hidelinks,onefignum,onetabnum]{siamart220329}

\usepackage{stmaryrd}
\SetSymbolFont{stmry}{bold}{U}{stmry}{m}{n}
\usepackage{amsmath,bbm,mathbbol,mathtools}
\usepackage{amssymb,amsfonts}
\usepackage{tabularx,lmodern}
\usepackage{color,soul,enumitem}
\usepackage{epsfig,epstopdf,color}
\usepackage{hyperref,graphics}
\usepackage{multirow}
\usepackage{algpseudocode}
\usepackage{flushend}% For balanced columns
\usepackage{verbatim}
\usepackage{caption}
\usepackage{graphicx}
\usepackage{arydshln}
\usepackage[level]{datetime}
\usepackage{float}
\usepackage{booktabs}
\usepackage{tabularx}
\usepackage{relsize}
\usepackage{ulem}
\usepackage{setspace}
\usepackage{epstopdf}
\usepackage{appendix}
\usepackage{bm}
\usepackage{listings}
\usepackage{subfigure}
\usepackage{parskip}
\usepackage{makecell}

\PassOptionsToPackage{ruled,nofillcomment,shortend}{algorithm2e}
\usepackage{algorithm2e}

\usepackage[T1]{fontenc}
\SetKwFor{For}{for}{do}{endfor for}%

\newcommand{\pb} { ~\underline{\perp}~}

\usepackage{color}
\usepackage{ulem}

\newcommand{\comm}[1]{{\color{black}#1}} %red
\newcommand{\revise}[1]{{\color{black}#1}} %blue
\newsiamremark{remark}{Remark}
\newsiamremark{example}{Example}
\algdef{SE}[DOWHILE]{Do}{doWhile}{\algorithmicdo}[1]{\algorithmicwhile\ #1}%

\allowdisplaybreaks[4]
\begin{document}
%\linenumbers

\title{A Bi-Orthogonal Structure-Preserving eigensolver \\ 
for large-scale linear response eigenvalue problem}

\author{
   Yu Li \footnotemark[2]
   \and 
   Zijing Wang \footnotemark[3]
   \and 
Yong Zhang\footnotemark[4] \footnotemark[5]
}
% \maketitle
% Authors: full names plus addresses.

\renewcommand{\thefootnote}{\fnsymbol{footnote}}
\footnotetext[1]{This work was partially supported by the National Natural Science Foundations of China under No.12271395 (Yu Li) and No.12271400 (Yong Zhang).
Zijing Wang was partially supported by 
the Strategic Priority Research Program of the Chinese Academy of Sciences (XDB0640000, XDB0640300, XDB0620203), 
National Natural Science Foundations of China (No.1233000214), 
Science Challenge Project (TZ2024009), 
State Key Laboratory of Mathematical Sciences, CAS.}
\footnotetext[2]{Coordinated Innovation Center for Computable Modeling in Management Science, Tianjin University of Finance and Economics, Tianjin, 300222, China.
   (\email{liyu@tjufe.edu.cn}).}
\footnotetext[3]{State Key Laboratory of Mathematical Sciences, Academy of Mathematics and Systems Science, Chinese Academy of Sciences, Beijing, 100190, School of Mathematical Sciences, University of Chinese Academy of Sciences, Beijing, 100049, China. 
  (\email{zjwang@lsec.cc.ac.cn}).}
\footnotetext[4]{Center for Applied Mathematics and KL-AAGDM, Tianjin University, Tianjin, 300072, China. 
  (\email{Zhang\_Yong@tju.edu.cn}).}
\footnotetext[5]{Corresponding author.}

\ifpdf
\hypersetup{
  pdftitle={A Bi-Orthogonal Structure-Preserving eigensolver \\ 
  for large-scale linear response eigenvalue problem},
  pdfauthor={Y. Li, Z. J. Wang, and Y. Zhang}
}
\fi

\pagestyle{myheadings}\thispagestyle{plain}
\markboth{Y. Li, Z. Wang and Y. Zhang}
{Bi-Orthogonal Structure-Preserving Eigensolver}

\maketitle

\begin{abstract}
The linear response eigenvalue problem, which arises from many scientific and engineering fields, is quite challenging numerically 
for large-scale sparse/dense system, 
especially when it has zero eigenvalues.
Based on a direct sum decomposition of biorthogonal invariant subspaces
and the minimization principles in the biorthogonal complement, 
using the structure of generalized nullspace, 
we propose a Bi-Orthogonal Structure-Preserving subspace iterative solver, which is stable, efficient, and of excellent parallel scalability. 
The biorthogonality is of essential importance and created by a modified Gram-Schmidt biorthogonalization (MGS-Biorth) algorithm.
We naturally deflate out converged eigenvectors by computing the rest eigenpairs in the biorthogonal complementary subspace 
without introducing any artificial parameters. 
When the number of requested eigenpairs is large, we propose a moving mechanism to compute 
them {\sl batch by batch} such that the projection matrix size is small and independent of the requested eigenpair number. 
For large-scale problems, one only needs to provide the matrix-vector product, thus waiving explicit matrix storage.
The numerical performance is further improved when the matrix-vector product is implemented using parallel computing.
Ample numerical examples are provided to demonstrate the stability, efficiency, and parallel scalability.

\end{abstract}


\section{Introduction}

In quantum physics and chemistry, the random phase approximation (RPA) is a theoretical framework used to 
elucidate the excitation states (energies) of physical systems 
during the investigation of collective motion in many-particle systems 
\cite{Ring2004Nuclear,Thouless1961Vibrational,Thouless1972Quantum}. For example, 
the characterization of excitation states and energies is frequently articulated through the utilization 
of the Bethe-Salpeter equation (BSE) 
\cite{Salpeter1951relativistic,Strinati1988Application}, and 
the Bogoliubov-de Gennes equations (BdG) are commonly employed for the description of 
collective motion and excitations \cite{Blase2020Bethe,Casida1995Time,Onida2002Electronic}. 
An essential problem within the realm of RPA pertains to 
the computation of 
the first few smallest positive eigenvalues and the corresponding eigenvectors
of the following eigenvalue problem
\begin{equation}\label{equ:bse}
	\begin{bmatrix}
		A & B \\ -B & -A
	\end{bmatrix}
	\begin{bmatrix}
		u \\ v
	\end{bmatrix}
	=\lambda 
	\begin{bmatrix}
		u \\ v
	\end{bmatrix},
	\end{equation}
	where $A, B\in\mathbb{R}^{n\times n}$ are both symmetric matrices and $u, v\in \mathbb R^{n}$ are column vectors.
By performing a change of variables
$u = y + x$, $v = y-x$,
the eigenvalue problem \eqref{equ:bse} can be transformed equivalently into
the linear response eigenvalue problem (LREP)
\begin{equation}\label{equ:xy_algebra}
	H
	\begin{bmatrix}
		y \\ x
	\end{bmatrix}
	=\lambda 
	\begin{bmatrix}
		y \\ x
	\end{bmatrix}
	\qquad\mbox{with }
	H = 
	\begin{bmatrix}
		O & K \\ M & O
	\end{bmatrix},
\end{equation}
where $K=A-B$ and $M=A+B$. 
In many physical problems \cite{Blase2020Bethe,Casida1995Time,Onida2002Electronic}, one of these two matrices is positive semi-definite.
Without loss of generality, we assume that matrix $K$ is symmetric positive semi-definite (SPSD)
and the other matrix $M$ is symmetric positive definite (SPD). 

Many methods have been proposed to compute LREP \eqref{equ:xy_algebra}. 
Bai et al. present minimization principles 
and develop 
a locally optimal block preconditioned 4-d conjugate gradient (LOBP4dCG) method
based on a structure-preserving subspace projection
\cite{Bai2012Minimization,Bai2013Minimization,Bai2014Minimization,Bai2016Linear,Rocca2012block}.
Lanczos-type algorithm, 
along with its variations and extensions \cite{Brabec2015Efficient,Shao2018structure,Teng2013Convergence,Teng2017block}
and the block Chebyshev-Davidson method \cite{Teng2016block} are also employed.
Meanwhile,
this nonsymmetric eigenvalue problem can be 
transformed into an eigenvalue problem that
involves the product of two matrices $K$ and $M$. 
Because $KM$ is self-adjoint with respect to the inner product induced by SPD matrix $M$,
defined as $x^{\top}My$ for $x,\ y\in\mathbb{R}^n$,
this product eigenvalue problem can be solved by
a modified Davidson algorithm and a modified locally optimal block preconditioned conjugate gradient (LOBPCG) algorithm
\cite{Vecharynski2017Efficient}.
A FEAST algorithm based on complex contour integration
for the linear response eigenvalue problem is presented in 
\cite{Teng2019feast}.
For dense structured LREP \eqref{equ:xy_algebra}, 
Shao and Benner et al. propose efficient
structure-preserving parallel algorithms with ScaLAPACK \cite{Benner2022efficientbse,Shao2016Structure}.

%In numerical computations, solving the eigenvalues and eigenvectors of a matrix is a central problem. 
For large-scale system, 
computing eigenpairs all at once is extremely challenging 
in terms of memory requirement, complexity and notorious slow convergence or even divergence.
The deflation mechanism is introduced to alleviate such difficulties by deducting converged eigenvectors from subsequent search subspace 
so that the following computation is not affected by converged eigenpairs.
In \cite{Bai2016Linear}, the deflation mechanism is implemented by 
a low-rank update to $K$ with parameter $\xi$,
which is determined using \textit{a priori} spectral distribution of $H$.
While, in this article, we deflate converged eigenpairs by using the inherent biorthogonal structure 
of eigenspaces associated with matrix $H$, waiving any artificial parameters.

%The eigenspaces of $H$ share a similar orthogonality named as biorthogonality \cite{Bai2012Minimization}.
To be specific, eigenvectors of $H$ associated with eigenvalues of different magnitudes, 
denoted as 
$\begin{bsmallmatrix} y_i \\ x_i \end{bsmallmatrix}$
with $i=1,2$, are biorthogonal to each other, i.e., $x_1^{\top}y_2=y_1^{\top}x_2=0$.
Therefore, it is possible to decompose the finite-dimension space $\mathbb{R}^{2n}$ into a direct sum of biorthogonal invariant subspaces, 
an analog of the orthogonal direct sum decomposition for real symmetric eigenvalue problems \cite{Li2023GCGE,Li2020parallel,Zhang2020generalized}.
Preserving such biorthogonality on the discrete level is very important to guarantee convergence.
With help of the stable biorthogonalization procedure, namely modified Gram-Schmidt biorthogonalization algorithm (MGS-Biorth), 
biorthogonality is created numerically and convergence is reached within about a few iterations.

Such biorthogonality provides a natural mechanism to deflate out converged eigenvectors, 
including those lying in the generalized nullspace,
by minimizing the trace functional in their biorthogonal complementary subspace. 
Thus, no artificial parameters are introduced in our deflation mechanism.
The biorthogonal structure is preserved well throughout the algorithm, therefore, we name it as Bi-Orthogonal Structure-Preserving solver (BOSP for short).

As for the biorthogonalization algorithm, the earliest literature dates back to 1950 \cite{Lanczos1950iteration}, 
where it was initially developed for solving eigenvalue problems by Lanczos.
Later, the Lanczos biorthogonalization algorithm, an extension of the symmetric Lanczos algorithm for nonsymmetric matrices,
has been extensively studied and applied to linear system and eigenvalue problems, and we refer to 
\cite{Parlett1985look,Saad1982Lanczos,Saad2003Iterative,Saad2011Numerical} for an incomplete list. 
In this paper, we utilize MGS-Biorth algorithm due to its excellent performance in suppressing rounding errors, 
and it is more stable than the classical Gram-Schmidt biorthogonalization algorithm (CGS-Biorth). 
%\comm{By employing MGS-Biorth, 
%the propagation of rounding errors is suppressed to a great extent, 
%thus ensuring that biorthogonality is well obtained at the discrete level.}

Spatial discretization of high-dimension partial differential equation problems often leads to large-scale sparse or dense problems,
\comm{for example, the finite element methods or spectral methods}, 
and the degrees of freedom might easily soar to millions or billions. Under such circumstances, 
explicit matrix storage is a very challenging task even for modern hardware architecture, therefore, it is imperative 
to develop a subspace iteration eigensolver where only the matrix-vector product is required. 
To this end, we design an interactive interface to help users provide matrix-vector product implementation, 
thus eliminating the need for any explicit matrix storage. 
The performance will be further elevated if the matrix-vector product is
numerical accessed with better efficiency, for example, using CPU or GPU parallel implementation.

In real applications, the number of requested eigenpairs, denoted as $n_{e}$, might be very large, 
for example, as large as hundreds or thousands in quantum physics \cite{Martin2021Stability,Shao2016Structure}.
\comm{We propose a novel moving mechanism to compute the eigenpairs 
  \textit{batch by batch}, where the number of eigenpairs in each batch, denoted as $n_b$, is smaller than $n_e$ and set as 
   $n_b =  \min\{n_{e}/5,150\}$ by default.}
Such strategy allows for incremental and sequential extraction of eigenpairs in batches,
%thus enables for handle large-scale problems.
thus, the projection matrix size is small, and, more importantly, 
independent of either degrees of freedom or the requested eigenpairs number.
Therefore, for large $n_e$ case, the memory requirement is greatly reduced and the overall efficiency is improved substantially.
Similar mechanism has been applied to large-scale symmetric eigenvalue problems, and we refer the readers to 
\cite{Hernandez2005SLEPc,Li2023GCGE,Li2020parallel,Zhang2020generalized} for more details.

%The primary objective is to develop a stable and efficient algorithm with excellent parallel scalability for large-scale LREP \eqref{equ:xy_algebra}.
%Our proposed algorithm is mainly based 
%MGS-Biorth algorithm, 
%therefore, 
%we name it as Bi-Orthogonal Structure-Preserving solver (BOSP for short).
%With the deflation and moving mechanism, it achieves convergence with excellent efficiency. 

The rest of the paper is organized as follows. 
In Section \ref{sec:background}, we introduce the basic theories of LREP and the biorthogonality of eigenvectors.
In Section \ref{sec:biorth_gcg}, we present MGS-Biorth algorithm and
the BOSP algorithm, together with various modern techniques to optimize its numerical performance. 
In Section \ref{sec:numerical}, 
we conduct comprehensive numerical tests to showcase the stability, efficiency, and parallel scalability.
Some concluding remarks are given in the last section.

\section{Background}\label{sec:background}

\comm{
In this section, we first recite some basic theoretical results on LREP \eqref{equ:xy_algebra},
which are also presented in \cite{Bai2012Minimization,Bai2013Minimization,Bai2016Linear}. 
Then,
a direct sum decomposition of biorthogonal invariant subspaces
for $\mathbb{R}^{2n}$
and the minimization principles in the biorthogonal complement
are given, 
which serve as the cornerstone of our algorithm.}

\subsection{Basic theory}\label{sec:biorth_property}

\begin{lemma}[Symmetric real eigenvalue distribution]
   \label{lem:eigReal}
	 %Suppose $K$ and $M$ are both SPD, 
	 %Let $\left\{\lambda; \begin{bsmallmatrix}y \\ x\end{bsmallmatrix}\right\}$ 
	 Let $\lambda$ be an eigenvalue of $H$ in LREP \eqref{equ:xy_algebra}, 
	 then 
	 \begin{itemize}
		 \item[\rm{(i)}] The eigenvalue $\lambda$ is a real number, i.e., $\lambda \in \mathbb{R}$.
		 \item[\rm{(ii)}] There exists a real eigenvector $\begin{bsmallmatrix}y \\ x\end{bsmallmatrix} \in \mathbb{R}^{2n}$ associated with the eigenvalue $\lambda$.
		 \item[\rm{(iii)}] %all eigenvalues come in pairs, i.e., 
			 $\left\{-\lambda; \begin{bsmallmatrix}-y\\~x\end{bsmallmatrix}\right\}$ 
			 is also an eigenpair of $H$.
 \item[\rm{(iv)}] 
   ${\rm rank}(\lambda I-H)
% ={\rm rank}((\lambda I-H)^2)
 %=n+{\rm rank}(\lambda^2 I-M^{1/2}KM^{1/2})$.
 =n+{\rm rank}(\lambda^2 I-KM)$.
   \end{itemize}
%We hereafter shall focus on the positive eigenvalues and denote them in increasing order as follows
%\begin{equation*}
%0 < \lambda_1 \leq \cdots \leq  \lambda_n.
%\end{equation*}
\end{lemma}
\begin{proof}
 The eigenvalues of $H$ are square roots of $KM$ because
\begin{equation*}
   {\rm det}(\lambda I -  H) = 
   {\rm det}
\begin{bmatrix}
	\lambda I & -K \\ -M & \lambda I
\end{bmatrix}
=
{\rm det}
\begin{bmatrix}
	%\lambda I & K \\ O & \lambda I - \frac{1}{\lambda}MK
	%\lambda I - \frac{1}{\lambda}KM & O \\  -M & \lambda I 
	O & -K+\lambda^2M^{-1} \\  -M & \lambda I 
\end{bmatrix}
={\rm det}(\lambda^2I - KM).
\end{equation*}
According to the fact that $KM = M^{-1/2}\left( M^{1/2} K M^{1/2}\right) M^{1/2}$  
is similar to a SPSD matrix $M^{1/2}KM^{1/2}$,
it follows that all eigenvalues of $KM$ are non-negative real numbers, 
which immediately implies that $\lambda^2\geq0$, i.e.,  $\lambda \in \mathbb R$.
Then there exists a real-valued eigenvector 
$\begin{bsmallmatrix}y\\x\end{bsmallmatrix} \in \mathbb{R}^{2n}$ 
associated with the eigenvalue $\lambda$.
Moreover, it is easy to check that 
$\begin{bsmallmatrix}-y\\x\end{bsmallmatrix}$ 
is also an eigenvector
with eigenvalue $-\lambda$.
It is straightforward to verify (iv).
\end{proof}

\iffalse
In many physical literature, instead of solving the full-matrix eigenvalue problem \eqref{equ:xy_algebra}, one is suggested  
to compute a compact-matrix problem, that is, one first solves 
\begin{equation}
   \label{equ:compact}
   KM y = \lambda^2 y, \mbox{ or, } M K x = \lambda^2 x,  \qquad \mbox{ for } x,y\in \mathbb R^{n},
\end{equation}
to obtain all the required eigenvalues together with their associated eigenvectors, 
then the associated problem $M y = \lambda x$ or $K x = \lambda y$ is solved 
to obtain the full eigenvector 
$\begin{bsmallmatrix}-y\\x\end{bsmallmatrix}$.

\vspace{0.2cm}
\begin{remark}\label{rem:compactForm}
In actual computation, 
solving the full-matrix problem \eqref{equ:xy_algebra} costs much less time than 
solving its compact version \eqref{equ:compact} for large-scale problem.
Firstly, 
when $K$ and $M$ are both SPD,
since we have the following identity
\begin{equation}
   {\rm cond}(H) = \sqrt{{\rm cond}(KM)} = \sqrt{{\rm cond}(MK)}, 
\end{equation}
the condition number of $H$ is usually much greater than one and much smaller than ${\rm cond}(MK)$.
Furthermore, it is worthy noticing that even though the length of an eigenvector 
of \eqref{equ:xy_algebra}
is twice that of \eqref{equ:compact}, 
the computational costs of matrix-vector product are the same. 
All in all, it is more efficient to solve the full-matrix problem.

\end{remark}
\fi

For sake of presentation simplicity, we extend the definition of biorthogonality \cite{Saad2003Iterative}
and introduce the following terminologies.
 
\begin{definition} 
Two vectors $\xi=\begin{bsmallmatrix}y\\x\end{bsmallmatrix},\
\eta = \begin{bsmallmatrix}z\\w\end{bsmallmatrix}\in \mathbb{R}^{2n}$ 
is said to be biorthogonal to each other if $y^{\top}w = z^{\top}x = 0$, 
and we shall denote such biorthogonality as $\xi \pb \eta$,
where $x,\ y,\ w,\ z\in\mathbb{R}^n$.
\end{definition}
%\vspace{0.4em}
\begin{definition}
Vector $\xi \in \mathbb R^{2n}$ is said to be biorthogonal to subspace $\mathcal{V}\subset\mathbb{R}^{2n}$
if it is biorthogonal to all elements of $\mathcal V$, i.e., $\xi \pb \eta, ~\forall~ \eta\in \mathcal V$, 
and we shall denote it as $\xi \pb \mathcal V$. Similarly, 
two subspaces $\mathcal{V}$ and $\mathcal{W}$ are biorthogonal, denoted as  $\mathcal V \pb \mathcal W$, 
if  $\xi \pb \eta, ~\forall~ \xi\in \mathcal V,\ \eta \in \mathcal W$. 
\end{definition}
%\vspace{0.4em}
\begin{definition}
	Matrice
	$\begin{bsmallmatrix} Y \\ X \end{bsmallmatrix}\in \mathbb{R}^{2n \times m}$
	are said to be biorthogonal 
	to subspace $\mathcal{V}\subset\mathbb{R}^{2n}$
	if  
	each column of 
	$\begin{bsmallmatrix} Y \\ X \end{bsmallmatrix}$
	is biorthogonal to subspace $\mathcal{V}\subset\mathbb{R}^{2n}$,
	denoted as  $\begin{bsmallmatrix} Y \\ X \end{bsmallmatrix} \pb \mathcal V$. 
\end{definition}
%\vspace{0.4em}
\begin{definition}
	Matrices $X=[x_1,\dots,x_m]$ and $\ Y=[y_1,\dots,y_m] \in \mathbb{R}^{n \times m}$ are said to be biorthogonal if  
	each column of 
	$\begin{bsmallmatrix} Y \\ X \end{bsmallmatrix}$
	is biorthogonal to the other columns and 
	$x_i^{\top}y_i = 1$.
	Equivalently, $X^{\top}Y = I_{m}$.
\end{definition}

Next, we derive the biorthogonal property of eigenvectors associated with eigenvalues of different magnitudes.
For convenience, we denote $\langle x,y\rangle := x^{\top}y$, $\forall ~x, y \in \mathbb{R}^n$. 
%and rewrite vector $\xi\in \mathbb R^{2n}$ 
%as $\xi = \begin{bsmallmatrix}y\\ x\end{bsmallmatrix}$ 
%with column vectors $y,\ x\in \mathbb{R}^{n}$ being the first and last half respectively,

\begin{lemma}[Biorthogonality of eigenvectors]\label{lem:eigPairing}
	Assume  
	$\left\{\lambda_i; \begin{bsmallmatrix} y_i \\ x_i \end{bsmallmatrix}\right\}$, 
	$i = 1,2$,
	 are eigenpairs of $H$ in LREP \eqref{equ:xy_algebra} 
	 with 
	 $|\lambda_1|\neq |\lambda_2|$, 
then the eigenvectors are biorthogonal, that is, 
	$\begin{bsmallmatrix} y_1 \\ x_1 \end{bsmallmatrix} \pb 
	\begin{bsmallmatrix} y_2 \\ x_2 \end{bsmallmatrix}$.
\end{lemma}

\begin{proof}
	Starting from equation \eqref{equ:xy_algebra}, we know that $MK x_i = \lambda_i^2 x_i$ and $KM y_i = \lambda_i^2 y_i$, for $i=1,2$. 
\comm{Then we have 
$\lambda_1^2 \langle y_2,x_1\rangle =\langle y_2, MK x_1\rangle  = \langle KM y_2,  x_1\rangle  =  \lambda_2^2 \langle y_2,x_1\rangle$; therefore,
$(\lambda_1^2 - \lambda_2^2)\langle y_2, x_1 \rangle = 0$, which implies 
$\langle y_2, x_1 \rangle = 0$.
Similarly, we obtain $\langle y_1, x_2 \rangle = 0$.}
\end{proof}

%\subsection{The generalized nullspace of $H$}

%In many physical applications, such as Bogoliubov-de Gennes equations (BdG), 
Since $K$ is a SPSD matrix, $H$ has zero eigenvalues.
Next, we shall first investigate the nullspace of $K$ and $H$, i.e.,
$\mathcal{N}(K) = \{x\in\mathbb{R}^{n}:Kx={\bf 0}\}$ and  $\mathcal{N}(H) = \{\xi\in\mathbb{R}^{2n}:H\xi={\bf 0}\}$,
%The convergence of non-zero eigenvalues, which are more concerned, depends on the approximation accuracy of 
\comm{and the generalized nullspace of $H$, which is defined by $\mathcal{N}(H^{2n}) = \{\xi\in\mathbb{R}^{2n}:H^{2n}\xi={\bf 0}\}$.}
%In this subsection, we shall focus on zero eigenvalue and the generalized nullspace.

\begin{lemma}[The generalized nullspace]
	\label{lem:zero}
	When ${\rm dim}(\mathcal{N}(K)) = r>0$ and 
	$\mathcal{N}(K) = {\rm span}\{x^0_1,\cdots,x^0_r\}$, 
	we have 
	\begin{itemize}
		\item[\rm{(i)}] The nullspace of $H$ is given explicitly as
				$\mathcal{N}(H) = {\rm span}\{
					\begin{bsmallmatrix}{\bf 0}\\ x^0_1\end{bsmallmatrix},\cdots,
				\begin{bsmallmatrix}{\bf 0}\\ x^0_r\end{bsmallmatrix} \}$  with $K x_i^0 = {\bf 0}$, $\forall~i = 1,\ldots, r$.
		\item[\rm{(ii)}] The generalized nullspace of $H$ is 
			$\mathcal{N}(H^{2n})=%\mathcal{N}(H^2)=
\mathcal{N}(H)+
				{\rm span}\{
					%\begin{bsmallmatrix}{\bf 0}\\ x^0_1\end{bsmallmatrix},\cdots,
					%\begin{bsmallmatrix}{\bf 0}\\ x^0_r\end{bsmallmatrix},
					\begin{bsmallmatrix}y^0_1\\ {\bf 0}\end{bsmallmatrix},\cdots,
					\begin{bsmallmatrix}y^0_r\\ {\bf 0}\end{bsmallmatrix} \}$
					with $My^0_i=x^0_i$, $\forall~i=1,\ldots,r$.
%		\item [\rm{(iii)}]
%			If $\left\{\lambda; \begin{bsmallmatrix}y\\x\end{bsmallmatrix}\right\}$ is an eigenpair of $H$,
%		we have $\lambda \langle y, x\rangle \geq 0$, and
%		\begin{equation*}
%			\langle y, x\rangle = 0 \iff \lambda = 0.
%		\end{equation*}
		\item [\rm{(iii)}]
            For non-zero eigenvalue $\lambda$ with eigenvector $\begin{bsmallmatrix}y\\x\end{bsmallmatrix}\in \mathbb{R}^{2n}$,
			we have %the following biorthogonal property:
	$\begin{bsmallmatrix} y^0_i \\ x^0_i \end{bsmallmatrix} \pb \begin{bsmallmatrix} y \\ x \end{bsmallmatrix}$, i.e.,
	$\langle y^0_i, x\rangle = \langle x^0_i, y \rangle = 0$, $\forall~i=1,\ldots,r$.
	\end{itemize}
\end{lemma}
\vspace{0.2cm}
\begin{proof}
For (i), it is obviously that ${\rm rank}(H) = {\rm rank}(K)+{\rm rank}(M)=n-r+n = 2n-r$.
So ${\rm dim}(\mathcal{N}(H)) = r$ and
$H\begin{bsmallmatrix}{\bf 0}\\ x^0_i\end{bsmallmatrix} = \begin{bsmallmatrix} Kx^0_i \\ {\bf 0}\end{bsmallmatrix}
= \begin{bsmallmatrix}{\bf 0}\\ {\bf 0}\end{bsmallmatrix}$, $\forall~i=1,\ldots,r$.

	For (ii), due to the fact that
		$H^2 
		=\begin{bsmallmatrix} KM & O \\ O & MK \end{bsmallmatrix}$ and 
		$H^3 =\begin{bsmallmatrix} O & KMK \\ MKM & O \end{bsmallmatrix}$,
	we have 
	${\rm rank}(H^2) = {\rm rank}(KM)+{\rm rank}(MK)= {\rm rank}(K)+{\rm rank}(K)=2(n-r)$, 
	and 
	${\rm rank}(H^3) = {\rm rank}(KMK)+{\rm rank}(MKM)=2(n-r)$,
	because 
	$${\rm rank}(KMK)={\rm rank}((M^{1/2}K)^{\top}M^{1/2}K) = {\rm rank}(M^{1/2}K)= {\rm rank}(K)=n-r.$$
	Moreover, it holds that
	 $\begin{bsmallmatrix} KM & O \\ O & MK \end{bsmallmatrix}
		\begin{bsmallmatrix}{\bf 0}\\ x^0_i\end{bsmallmatrix} = 
		\begin{bsmallmatrix}{\bf 0}\\ {\bf 0}\end{bsmallmatrix}$ and 
   $\begin{bsmallmatrix} KM & O \\ O & MK \end{bsmallmatrix}
	 \begin{bsmallmatrix}y^0_i\\ {\bf 0}\end{bsmallmatrix} = \begin{bsmallmatrix}{\bf 0}\\ {\bf 0}\end{bsmallmatrix}$,
		$\forall~i=1,\ldots,r$.
	Then, based on 
	${\rm rank}([y^0_1,\cdots,y^0_r])={\rm rank}([x^0_1,\cdots,x^0_r])=r$, 
	(ii) is proved.

%	For (iii), 
%	it is clear that $\lambda \langle y, x\rangle = y^{\top}My\geq0$.
%	If $\lambda=0$, it holds that
%	$\langle y, x \rangle = 0$ since $y={\bf 0}$.
%     Since $\begin{bsmallmatrix}y\\x\end{bsmallmatrix}$ is an eigenvector, if $\langle y, x \rangle = 0$, we have $0=\langle y, \lambda x \rangle = \langle y, M y \rangle$, 
%		 then $y={\bf 0}$ and $x\neq {\bf 0}$.
%	Provided that $My = {\bf 0}= \lambda x$, it is clear that $ \lambda =0$.

      For (iii),
	due to the facts that $Kx=\lambda y$ and $My=\lambda x$ with $\lambda\neq0$, it follows that
		$\langle x^0_i, y \rangle
		=\langle x^0_i, \frac{1}{\lambda} Kx\rangle 
		=\langle Kx^0_i, \frac{1}{\lambda}x\rangle = 0$ and 
		$\langle y^0_i, x\rangle 
		=\langle y^0_i, \frac{1}{\lambda} My\rangle 
		=\langle My^0_i, \frac{1}{\lambda} y\rangle
		=\langle x^0_i, \frac{1}{\lambda} y\rangle =0$.
\end{proof}

\subsection{Direct sum decomposition and minimization principle}
\label{sec:dsd_mp}
%In summary, we present a direct sum decomposition of biorthogonal invariant subspaces of $H$
%The similar results 

\comm{
The Jordan canonical form of $H$ is given in \cite{Bai2012Minimization} (Theorem 2.3), 
and it is equivalent to a direct sum decomposition of $\mathbb{R}^{2n}$ into $2(n-r)+r$ invariant subspaces, 
with each invariant subspace being associated with a Jordan block.
Here, according to the biorthogonal property of eigenvectors of $H$,
we decompose $\mathbb{R}^{2n}$ into $(n-r)+1$ 
\textbf{biorthogonal} invariant subspaces (see Theorem \ref{thm:dsd_biorth_is}),
where $\mathcal{V}_0$
corresponds to $r$ Jordan blocks of order two with eigenvalue $0$
(the generalized nullspace),
and $\mathcal{V}_i$ 
corresponds to $2$ Jordan blocks of order one
with eigenvalues $\pm\lambda_i$ for $i=1,\dots,n-r$.
In essence, the Jordan canonical form algebraically formalizes the invariant subspace decomposition, 
while the decomposition provides a geometric understanding.
}

\begin{theorem}[\textbf{Direct sum decomposition}]
	\label{thm:dsd_biorth_is}
	Let $H\in\mathbb{R}^{2n\times 2n}$ be given in \eqref{equ:xy_algebra}
	and ${\rm dim}(\mathcal{N}(H)) = r>0$,
	then we have 
\begin{equation}
	\label{equ:decomp_invar_subsp}
	\mathbb{R}^{2n} = \mathcal{V}_0\oplus \mathcal{V}_{1}\oplus \cdots \oplus \mathcal{V}_{n-r},
\end{equation}
where %the generalized nullspace
$\mathcal{V}_0
	={\rm span}\{
		\begin{bsmallmatrix} {\bf 0} \\ x_1^0 \end{bsmallmatrix},\cdots,
		\begin{bsmallmatrix} {\bf 0} \\ x_r^0 \end{bsmallmatrix},
		\begin{bsmallmatrix} y_1^0 \\ {\bf 0} \end{bsmallmatrix},\cdots,
	\begin{bsmallmatrix} y_r^0 \\ {\bf 0} \end{bsmallmatrix}\}$
and 
$\mathcal{V}_i 
 = {\rm span}\{
	\begin{bsmallmatrix} {\bf 0} \\ x_i \end{bsmallmatrix},
\begin{bsmallmatrix} y_i \\ {\bf 0} \end{bsmallmatrix}\}$,
for $i=1,\ldots,n-r$,
are the invariant subspaces of $H$
satisfying 
\begin{equation}
	\label{equ:xy_zero}
H
\begin{bmatrix} {\bf 0} \\ x_i^0 \end{bmatrix}
=
\begin{bmatrix} {\bf 0} \\ {\bf 0} \end{bmatrix},\qquad
H
\begin{bmatrix} y_i^0 \\ {\bf 0} \end{bmatrix}
=
\begin{bmatrix} {\bf 0} \\ x_i^0 \end{bmatrix},\qquad \forall~i=1,\ldots,r,
\end{equation}
and 
\begin{equation}
	\label{equ:xy_non_zero}
%\begin{bmatrix} O & K \\ M & O \end{bmatrix}
H
\begin{bmatrix} y_i \\ x_i \end{bmatrix}
= \lambda_i
\begin{bmatrix} y_i \\ x_i \end{bmatrix},\qquad
%\begin{bmatrix} O & K \\ M & O \end{bmatrix}
H
\begin{bmatrix} -y_i \\ x_i \end{bmatrix}
=-\lambda_i
\begin{bmatrix} -y_i \\ x_i \end{bmatrix},
\qquad \forall~i=1,\ldots,n-r.
\end{equation}
Here, $0<\lambda_1\leq\lambda_2\leq\cdots\leq\lambda_{n-r}$ are all positive eigenvalues,
and 
the biorthogonal property
\begin{equation}
	\label{equ:XYI}
	[x_1^0,\cdots,x_r^0,x_1,\cdots,x_{n-r}]^{\top}
	[y_1^0,\cdots,y_r^0,y_1,\cdots,y_{n-r}]=I_n
\end{equation}
holds, i.e., $\mathcal{V}_i \pb \mathcal{V}_j$, $\forall~i\neq j$.
\end{theorem}
\begin{proof}
	Let the nullspace of $K$ be ${\rm span}\{\widetilde{x}^0_1,\cdots, \widetilde{x}^0_r \}$,
	then there exist $\widetilde{y}^0_1,\ldots,\widetilde{y}^0_r\in\mathbb{R}^n$
satisfying $M \widetilde{y}_i^0 = \widetilde{x}_i^0$, $\forall~i=1,\ldots,r$.
Thus, we obtain
\begin{equation*}
	[\widetilde{x}^0_1,\cdots,\widetilde{x}^0_r]^{\top}
	[\widetilde{y}^0_1,\cdots,\widetilde{y}^0_r]
	=
	[\widetilde{y}^0_1,\cdots,\widetilde{y}^0_r]^{\top}
	M
	[\widetilde{y}^0_1,\cdots,\widetilde{y}^0_r]
	= C^{\top}C,
\end{equation*}
where $C\in\mathbb{R}^{r\times r}$ is an invertible matrix.
It is easy to check that 
\begin{equation}
	\label{equ:V0_basis}
	[{x}^0_1,\cdots,{x}^0_r]=
	[\widetilde{x}^0_1,\cdots,\widetilde{x}^0_r]C^{-1},\quad
	[{y}^0_1,\cdots,{y}^0_r]=
	[\widetilde{y}^0_1,\cdots,\widetilde{y}^0_r]C^{-1}
\end{equation}
satisfy equation \eqref{equ:xy_zero} and 
$[{x}^0_1,\cdots,{x}^0_r]^{\top}[{y}^0_1,\cdots,{y}^0_r]=I_r$.
Similarly, using Lemma \ref{lem:eigReal} and \ref{lem:eigPairing}, we can prove that  
there exist 
	$x_1,\ldots,x_{n-r}$ and $y_1,\cdots,$
	$y_{n-r}$ satisfying
	equation \eqref{equ:xy_non_zero}
	and 
	$
	[x_1,\cdots,x_{n-r}]^{\top}
	[y_1,\cdots,y_{n-r}]=I_{n-r}.
	$
	According to Lemma \ref{lem:zero} (iii), equation \eqref{equ:XYI} holds.
	Therefore, the direct sum decomposition \eqref{equ:decomp_invar_subsp}
	is proved.
\end{proof}

\comm{
Based on the above direct sum decomposition \eqref{equ:decomp_invar_subsp},
we construct the following minimization principles
in the biorthogonal complement of 
$\mathcal{V}_{0}\oplus \cdots \oplus \mathcal{V}_{{\ell}}$, 
and provide a concise proof.
}

\begin{theorem}[\textbf{Minimization principles}]
	\label{thm:biorth_compl}
	For given $k=1,\ldots,n-r$ and $\ell = 0,1,\ldots,n-r-k$, we have
	\begin{equation}
		\label{equ:lambda}
		\sum_{i=1}^k\lambda_{\ell+i}
		=\min_{\scriptstyle 
			\substack{
\begin{bsmallmatrix}V\\U\end{bsmallmatrix}\pb
\mathcal{V}_{0}\oplus \cdots \oplus \mathcal{V}_{{\ell}}\\
			{U}^{\top}{V}=I_k} 
		}
		\frac{1}{2}
		{\rm trace}\big({U}^{\top}K{U} + {V}^{\top}M{V}\big).
	\end{equation}
\end{theorem}

\begin{proof}
We shall only prove the case in which all the eigenvalues are simple,
and the case of multiple eigenvalues can be derived analogously.
According to 
the direct sum decomposition of biorthogonal invariant subspaces of $H$
in Theorem \ref{thm:dsd_biorth_is},
for any $\begin{bsmallmatrix}V\\U\end{bsmallmatrix}\in\mathbb{R}^{2n\times k}$,
there exist
$C_1,\ D_1\in\mathbb{R}^{(r+\ell)\times k}$ and 
$C_2,\ D_2\in\mathbb{R}^{(n-r-\ell)\times k}$
such that
\begin{align*}
	V&=[y_1^0,\cdots,y_r^0,y_1,\cdots,y_{\ell}]D_1+[y_{\ell+1},\cdots,y_{n-r}]D_2,\\
	U&=[x_1^0,\cdots,x_r^0,x_1,\cdots,x_{\ell}]C_1+[x_{\ell+1},\cdots,x_{n-r}]C_2.
\end{align*}
If $\begin{bsmallmatrix}V \\ U\end{bsmallmatrix} \pb	\mathcal{V}_{0}\oplus \cdots \oplus \mathcal{V}_{\ell}$,
then 
$[x_1^0,\cdots,x_r^0,x_1,\cdots,x_{\ell}]^{\top}V$ and $[y_1^0,\cdots,y_r^0,y_1,\cdots,y_{\ell}]^{\top}U$
are both zero matrices.
Due to the biorthogonal property \eqref{equ:XYI}
and $U^{\top}V=I_k$,
we obtain that $C_1$ and $D_1$ are both zero matrices, i.e., each column of $\begin{bsmallmatrix} V \\ U \end{bsmallmatrix}$
belong to $\mathcal{V}_{\ell+1}\oplus \cdots \oplus \mathcal{V}_{n-r}$, and $C_2^{\top}D_2=I_{k}$.  
Then we have 
\begin{equation}
	\label{equ:trace_1}
	\begin{split}
		&\min_{\scriptstyle 
		\substack{
			\begin{bsmallmatrix}V\\U\end{bsmallmatrix}\pb
			\mathcal{V}_{0}\oplus \cdots \oplus \mathcal{V}_{\ell}\\
		{U}^{\top}{V}=I_k} 
	}
	\frac{1}{2}
	{\rm trace}\big({U}^{\top}K{U} + {V}^{\top}M{V}\big)\\
	=&
	\min_{\scriptstyle 
		\substack{
		{C_2}^{\top}{D_2}=I_k} 
	}
	\frac{1}{2}
	{\rm trace}\big(
		{C_2}^{\top} \Lambda {C_2} 
		+ {D_2}^{\top} \Lambda {D_2}\big).
	\end{split}
\end{equation}
where $\Lambda = {\rm diag}\{{\lambda}_{\ell+1},\cdots,{\lambda}_{n-r}\}$.
According to first order necessary optimality condition, 
any extreme points $C_2$ and $D_2$ satisfies the Lagrangian system:
\begin{equation*}
		{\Lambda}C_2-D_2\Theta^{\top}  = O,\
		{\Lambda}D_2-C_2\Theta= O,\
		C_2^{\top}D_2 = I_k,
	\end{equation*}
where $\Theta=(\theta_{ij})_{k\times k}$
and $\theta_{ij}$ is the Lagrangian multiplier.
It can be directly verified that the column space of $C_2$ and $D_2$
are the same $k$-dimensional invariant subspaces of ${\Lambda}^2$.
Specifically, there exist two invertible matrices
$\widehat{C},\ \widehat{D}\in\mathbb{R}^{k\times k}$
that satisfy 
$C_2 =[e_{i_1},\cdots,e_{i_k}]\widehat{C}$,
$D_2 =[e_{i_1},\cdots,e_{i_k}]\widehat{D}$ and
$\widehat{C}^{\top}\widehat{D} = I_k$,
%	D_2 =[e_{j_1},\cdots,e_{j_k}]\widehat{D}$,
where $e_{i_p}$ is the $i_p$-th column of the $(n-r-\ell)$-order identity matrix,
and the sequence of integers 
$\{i_1,\dots,i_k\}$ satisfy
$1\leq i_1< \dots< i_k\leq (n-r-\ell)$. %and $1\leq j_1< \dots< j_k\leq (n-r-\ell)$ respectively.
%Due to $C_2^{\top}D_2 = I_k$,
%we have $i_p=j_p$ for $p=1,\dots,k$ and $\widehat{C}^{\top}\widehat{D} = I_k$.

Performing the singular value decomposition on matrix $\widehat{C} = Q\Sigma P^{\top}$,
we have $\widehat{D} = \widehat{C}^{-\top}= Q\Sigma^{-1} P^{\top}$.
It is noticed that 
\begin{equation*}
	\begin{split}
&	\frac{1}{2}
	{\rm trace}\big(
		{C_2}^{\top} \Lambda {C_2} 
		+ {D_2}^{\top} \Lambda {D_2}\big)=
	\frac{1}{2}
	{\rm trace}\big(
		\widehat{C}^{\top} \widehat{\Lambda} \widehat{C} 
		+ \widehat{D}^{\top} \widehat{\Lambda} \widehat{D}\big)
	\\
=&\frac{1}{2}
	{\rm trace}\big(
		(\Sigma^2+\Sigma^{-2}) Q^{\top}
		\widehat{\Lambda}
Q\big)
	\geq
	{\rm trace}\big(
		Q^{\top}
		\widehat{\Lambda}
Q\big)
=
	{\rm trace}\big(
		\widehat{\Lambda}
	\big),
	\end{split}
\end{equation*}
where $\widehat{\Lambda} = {\rm diag}\{{\lambda}_{\ell+i_1},\cdots,{\lambda}_{\ell+i_k}\}$.
So equation \eqref{equ:trace_1} has a lower bound $\sum_{i=1}^{k}\lambda_{\ell+i}$
which is attained when $C_2=D_2 = \begin{bsmallmatrix} I_k \\ O\end{bsmallmatrix}$.
To sum up, equation \eqref{equ:lambda} holds.
\end{proof}

\comm{
\begin{remark}[Deflation mechanism]
According to Theorem \ref{thm:biorth_compl},
we can compute $\lambda_{\ell+i}$ for $i=1,\dots,k$
in the biorthogonal complement
of the invariant subspaces
$\mathcal{V}_{0}\oplus \cdots \oplus \mathcal{V}_{l}$. 
Such a deflation mechanism 
is similar to that 
for symmetric eigenvalue problems in \cite{Hernandez2005SLEPc,Li2023GCGE}, 
and \textsl{does not} require any \textit{a priori} spectral distribution of $H$.
\end{remark}
}

\begin{remark}[Numerical computation of generalized nullspace]
\label{rem:V0}
As stated in the minimization principles \eqref{equ:lambda}, 
the calculation of positive eigenvalues is done in the biorthogonal complementary invariant subspace associated 
with the generalized nullspace $\mathcal{V}_0$, 
so the numerical approximation $\mathcal{V}_0$ is of great importance and should be computed with great accuracy in the very beginning.
As shown in Theorem \ref{thm:dsd_biorth_is},
all zero eigenvalues and their corresponding eigenvectors $\{\widetilde{x}^0_i\}_{i=1}^r$ of 
the symmetric matrix $K$ can be obtained using existing symmetric eigensolvers,
such as Lanczos algorithm and \revise{subspace projection method}
\cite{Hernandez2005SLEPc,Li2023GCGE,Li2020parallel,Saad2011Numerical,Zhang2020generalized}.
Then we compute $\widetilde{y}^0_i$ by solving the SPD linear system $M\widetilde{y}^0_i=\widetilde{x}^0_i$
for $i=1,\dots,r$ with conjugate gradient method. %Then we obtain a basis of $\mathcal{V}_0$. 
%%Thus, according to \eqref{equ:V0_basis}, we obtain a basis of $\mathcal{V}_0$. 

If the numerical approximation of the generalized nullspace $\mathcal{V}_0$ is not accurate enough,
the computation of its biorthogonal complement may fail to meet required tolerance and it shall affect
the convergence and accuracy of the non-zero eigenvalues.
\end{remark}

%\iffalse
%From a numerical perspective, 
%if the approximation accuracy of the generalized nullspace of $H$, i.e., $x^0_1$ and $y^0_1$, 
%is insufficent, 
%In this example, setting $x^0_1=[1,\epsilon]^{\top}$,
%we can get $y^0_1=[1,\epsilon]^{\top}$ by solving $My^0_1 = x^0_1$.
%Then, 
%for the smallest positive eigenvalue, 
%the minimization principles \eqref{equ:lambda}
%can be expressed as 
%\begin{equation*}
%	\begin{split}
%		\lambda_{1}
%	&=\min_{\scriptstyle 
%		\substack{
%			v^{\top}x^0_1=u^{\top}y^0_1=0\\
%		{u}^{\top}{v}=1} 
%	}
%	\frac{1}{2}
%	{\rm trace}({u}^{\top}K{u} + {v}^{\top}M{v})\\
%	&=\min_{\scriptstyle 
%		\substack{
%			v_1+\epsilon v_2=u_1+\epsilon u_2=0\\
%			u_1v_1+u_2v_2=1
%		} 
%	}
%	\frac{1}{2} (u_2^2+v_1^2+v_2^2) = ({1+\epsilon^2})^{-1/2}.
%	\end{split}
%\end{equation*}
%\begin{equation*}
%	u_1 =-\epsilon(1+\epsilon^2)^{-1/4} ,\ u_2 = (1+\epsilon^2)^{-1/4},\ 
%	v_1 =-\epsilon(1+\epsilon^2)^{-3/4} ,\ v_2 = (1+\epsilon^2)^{-3/4}.
%\end{equation*}
%%When $r=1$, $k=1$ and $\ell=0$, 
%%Here, the constrain $v^{\top}x^0_1=u^{\top}y^0_1=0$ is crucial.
%\fi

\section{BOSP algorithm}\label{sec:biorth_gcg}

In this section, 
we shall first describe the Bi-Orthogonal Structure-Preserving algorithm 
and present the detailed eigensolver.

\subsection{Biorthogonalization algorithm}
\label{sec:mgs}
Let us first recall the {CGS-Biorth} and 
{MGS-Biorth} algorithm 
\cite{Kohaupt2014Introduction,Ruhe1983Numerical,Saad2003Iterative}.
We have at hand 
two vectors $x,\ y\in\mathbb{R}^{n}$,
two biorthogonal matrices 
$P=[p_1,\cdots,p_{\ell-1}],\ Q=[q_1,\cdots,q_{\ell-1}]\in\mathbb{R}^{n\times {(\ell-1)}}, {\ell}<n$ satisfying that matrix $[P,x]^{\top}[Q,y]$ is invertible.
Our goal is to expand column spaces of $P$ and $Q$ by adding $x$ and $y$ respectively.
That is to construct two vectors $p_{\ell},q_{\ell}\in \mathbb{R}^n$ from $x,y$ such that
$P^{\top}q_{\ell} = Q^{\top}p_{\ell}={\bf 0}$ and $p_{\ell}^{\top}q_{\ell}=1$. 

Similar to the classical and modified Gram-Schmidt orthogonalization procedures, 
we obtain the CGS-Biorth and MGS-Biorth algorithms as follows\\
\begin{minipage}{0.48\textwidth}
\begin{equation*}
\begin{split}
&{\hspace{0.1cm}\mbox{CGS-Biorth}}\\
&\tilde{x} = x,\ \ \tilde{y} = y\\	
&\mbox{for } j=1 \mbox{ to } {\ell}-1 \\
&\ \ \ \ \tilde{x} = \tilde{x}- p_j \langle q_j,{x} \rangle\\
&\ \ \ \ \tilde{y} = \tilde{y}- q_j \langle p_j,{y} \rangle\\
&\mbox{end }
\end{split}
\end{equation*}
\end{minipage}
%\hfill
\begin{minipage}{0.48\textwidth}
\begin{equation*}
\begin{split}
&\hspace{0.1cm}{\mbox{MGS-Biorth}}\\
&\tilde{x} = x,\ \ \tilde{y} = y\\	
&\mbox{for } j=1 \mbox{ to } {\ell}-1 \\
&\ \ \ \ \tilde{x} = \tilde{x}- p_j \langle q_j,\tilde{x} \rangle\\
&\ \ \ \ \tilde{y} = \tilde{y}- q_j \langle p_j,\tilde{y} \rangle\\
&\mbox{end }
\end{split}
\end{equation*}
\end{minipage}\\
Then we set $p_{\ell} = \mbox{sign}(\eta)\tilde{x}/\sqrt{|\eta|},\, q_{\ell} = \tilde{y}/\sqrt{|\eta|},\, \eta =\tilde{x}^{\top}\tilde{y}$ such that $p_{\ell}^{\top}q_{\ell}=1.$ 
We obtain two biorthogonal matrices $\widetilde{P}=[P,p_{\ell}],\widetilde{Q}=[Q,q_{\ell}]$, i.e., $\widetilde{P}^\top \widetilde{Q}=I_{\ell}$.
It is obvious that both algorithms share the same computational complexity ($8n(\ell-1)$ flops), 
but their numerical stability performance, regarding the rounding off pollution on biorthogonality, are quite different with modified algorithm being more robust.

Taking the roundoff accumulation into account and 
assuming the rounding errors are the sole cause of biorthogonality loss, 
we can formulate the biorthogonality property as
$P^{\top}Q = I_{\ell-1}+L+U,$
where $L$/$U$ is a strictly lower/upper triangular matrix with all
the elements being as small as the roundoff.
Note that vectors $\tilde{x},\tilde{y}$ can be written in the following way
\begin{equation}
\label{BiorthCoef}
\tilde{x} =  x-Pr, \qquad \tilde{y} = y-Qs, \quad r,s\in \mathbb{R}^{\ell-1}.
\end{equation}
To find coefficients $r,s$ such that  $Q^{\top}\tilde{x} = P^{\top}\tilde{y}={\bf 0}$, 
one actually solves the following linear systems
\begin{equation}
	\label{equ:ls_biorth}
	Q^{\top}P r = Q^{\top}x,\quad  \mbox{ and }\quad  P^{\top}Q s = P^{\top}y.
\end{equation}
As suggested in \cite{Ruhe1983Numerical}, 
the classical and modified biorthogonalization algorithms
correspond to applying Gauss-Jacobi and Gauss-Seidel iterations with zero initial guesses $r^{(0)}=s^{(0)}={\bf 0}$ for one iteration ending up with $r=r^{(1)},s =s^{(1)}$, given explicitly\\
\begin{minipage}{0.45\textwidth}
\begin{equation*}
\begin{split}
	r^{(1)}&= Q^{\top}x, \\
	s^{(1)}&= P^{\top}y,\\
	&\hspace{-0.8cm}{\mbox{(Gauss-Jacobi)}}
\end{split}
\end{equation*}
\end{minipage}
\begin{minipage}{0.45\textwidth}
\begin{equation*}
\begin{split}
	r^{(1)}&= (I_{\ell-1}+U^{\top})^{-1}Q^{\top}x, \\
	s^{(1)}&= (I_{\ell-1}+L)^{-1}P^{\top}y,\\
	&\hspace{0.1cm}{\mbox{(Gauss-Seidel)}}
\end{split}
\end{equation*}
\end{minipage}\\[0.2em]
as combination coefficients for the classical and modified algorithm respectively. 

It is known that if $P^{\top}Q$ has some properties
as described in Corollary 4.6 \cite{Varga2000Matrix}
or Corollary 4.16 \cite{Kress2012Numerical}, 
the spectral radius of Gauss-Seidel iteration matrix, i.e.,
$-(I_{\ell-1}+U^{\top})^{-1}L^{\top}$ or $-(I_{\ell-1}+L)^{-1}U$,
is square of that for Gauss-Jacobi iteration matrix, i.e.,
$-(L+U)^{\top}$ or $-(L+U)$.
In other words, the modified algorithm always performs better than classical algorithm.

%Similar to the classical Gram-Schmidt orthogonalization method procedure, 
%the CGS-Biorth algorithm may suffer from severe numerical instability, 
%resulting in a substantial loss of biorthogonality due to the roundoff accumulation. 
%To alleviate the pollution of roundoff accumulation and to achieve a better numerical stability performance,
%we introduce the MGS-Biorth algorithm.

To be more specific, we present a detailed step-by-step description of MGS-Biorth in Algorithm \ref{alg:biorth}. 
%Without considering the rounding errors,
%it is easy to check that $P$ and $Q$ 
%obtained by Algorithm \ref{alg:biorth},
%are biorthogonal to each other, i.e., $P^{\top}Q=I_{\ell}$, and ${\rm span}(P) = {\rm span}(X)$, ${\rm span}(Q) = {\rm span}(Y)$.

\begin{algorithm}[!htbp]
	\setstretch{1.0}
\caption{Modified Gram-Schmidt biorthogonalization} 
\label{alg:biorth}
	\KwIn{ $X=[x_1,\cdots,x_{m}],\ Y=[y_1,\cdots,y_{m}]\in\mathbb{R}^{n\times {m}}$ with ${m}\leq n$,
	where $X^{\top}Y$ is invertible.}
	\KwOut{ $P=[p_1,\cdots,p_{m}]$ and $Q=[q_1,\cdots,q_{m}]$ satisfying $P^{\top}Q=I_{m}$.}
	\vspace{0.2cm}
	$\eta = {x_1^{\top}y_1}$\;
	$r_{11} = \mbox{sign}(\eta)\sqrt{|\eta|}$ and $s_{11} = \sqrt{|\eta|}$\;
	$p_1=x_1/r_{11}$ and $q_1=y_1/s_{11}$\;
	\For { ${\ell} = 2 : {m}$}{
		$p_{\ell}=x_{\ell}$ and $q_{\ell}=y_{\ell}$\;
		\For { $j = 1 : {\ell}-1$}{
			$r_{j\ell} = {q_j^{\top}p_{\ell}}$\;
 	$p_{\ell} = p_{\ell}-{r}_{j\ell}{p}_{j}$\;
 	${s}_{j\ell} = p_{j}^{\top} q_{\ell}$\;
 	$q_{\ell} = q_{\ell}-{s}_{j\ell}{q}_{j}$\;
	}
	$\eta = {p_{\ell}^{\top}q_{\ell}}$\;
	$r_{\ell\ell} = \mbox{sign}(\eta)\sqrt{|\eta|}$ and $s_{\ell\ell} = \sqrt{|\eta|}$\;
	$p_{\ell}=p_{\ell}/r_{\ell\ell}$ and $q_{\ell}=q_{\ell}/s_{\ell\ell}$\;
	}
\end{algorithm}

%\
\vspace{0.2cm}
\begin{remark}[Exception handling of ill-conditioned matrices]
	\label{rem:biorth_ill_conditioned}
  In Algorithm \ref{alg:biorth},
  when $X^{\top}Y$ is ill-conditioned,
there exists some $s_{\ell\ell}$ quite close to zero, then the biorthogonalization procedure will be unstable.
In practice, \revise{such vectors $x_{\ell}$ and $y_{\ell}$ are simply discarded and the iteration continues.} Consequently, the number of columns of $P$ and $Q$ might be less than ${m}$.
\end{remark}

When applied to matrix $P,Q$, the CGS-Biorth and MGS-Biorth algorithm have the same computational complexity, that is, $4\,n{m}^2$ flops.
The MGS-Biroth algorithm proves to be much more stable in numerical practice.
Here we carry out a detailed numerical comparison to show the biorthogonal preservation capabilities of both methods.
Matrices $X,\ Y\in\mathbb{R}^{n\times {m}}$ are constructed using the first ${m}$ columns of Hilbert matrix of order $n \times n$,
and Lauchli matrix \cite{Higham1995test} of order $n\times(n-1)$ with ${m}=n/2$. 
We adopt the biorthogonal residual $\Vert {P}^{\top}{Q}-I_{m}\Vert_2$ to measure the loss of biorthogonality,
where $P$ and $Q$ are matrices obtained from $X$ and $Y$, and $\Vert  \cdot \Vert_2$ is the discrete $l^2$ norm.
The condition number of matrix $A$, denoted by $\kappa(A)$,
is defined as
$\kappa(A) := \max_{\Vert v\Vert_2=1}\Vert Av\Vert_2/\min_{\Vert v\Vert_2=1}\Vert Av\Vert_2$.

\revise{Table \ref{tab:biorth_ill_conditioned} presents the condition numbers of matrices $X$, $Y$, $X^\top Y$ 
and the biorthogonal residual for different $n$,
from which we can observe that both methods experience a loss of biorthogonality as the condition number of
 $X^{\top}Y$ increases.}
\revise{Similar to the orthogonalization procedure, the MGS-Biorth algorithm is capable of maintaining a relatively high level of biorthogonality, 
even when the condition number is very large.}
Furthermore, from our extensive numerical observations not shown here, we conjecture
% the estimation
\revise{the following estimation holds for MGS-Biorth:}
\begin{equation} \label{equ:bcgs}
   \Vert P^{\top}Q-I_{m}\Vert_2
	=\mathcal{O(\epsilon)}\sqrt{\kappa(X^{\top}Y)},
\end{equation}
where $\epsilon$ denotes the rounding error. 
A comprehensive theoretical investigation is now going on and we shall report it later in another separate article.
\begin{table}[!htbp]
	\def\temptablewidth{1.0\textwidth}\tabcolsep 10pt
\centering
	{\rule{\temptablewidth}{1pt}}
	\resizebox{\linewidth}{!}{
	% \begin{tabularx}{\temptablewidth}{@{\extracolsep{\fill}}r|ccc|ccc}
	\begin{tabular}{c|ccc|ccc}
    $n$ &  $\kappa(X)$  & $\kappa(Y)$ & $\sqrt{\kappa(X^{\top}Y)}$ & \makecell{CGS-Biorth \\ $\Vert {P}^{\top}{Q}-I_{m}\Vert_2$} & \makecell{MGS-Biorth \\ $\Vert {P}^{\top}{Q}-I_{m}\Vert_2$} \\
 %& $\frac{\|P^{\top}Q-I_k\|_2}{\sqrt{\kappa(X^{\top}Y)}}$\\
		\hline\rule{0pt}{15pt}
\,4 & 1.32E+01 & 1.41E+03 & 1.34E+04 & 2.42E-12 & 2.42E-12 \\   
~8 & 4.42E+03 & 2.00E+03 & 4.82E+05 & 1.67E-07 & 2.32E-08 \\     
12 & 1.67E+06 & 2.44E+03 & 1.67E+07 & 3.48E-02 & 8.79E-07 \\
16 & 6.50E+08 & 2.82E+03 & 5.74E+08 & 4.98E+01 & 2.03E-04 \\
20 & 2.57E+11 & 3.16E+03 & 1.96E+10 & 9.18E+01 & 3.89E-03 \\
% 4  & 1.3269E+01 & 1.4142E+03 & 1.3422E+04 & 2.4295E-12 & 2.4295E-12 \\   
% 8  & 4.4285E+03 & 2.0000E+03 & 4.8236E+05 & 1.6775E-07 & 2.3257E-08 \\     
%12  & 1.6705E+06 & 2.4495E+03 & 1.6734E+07 & 3.4837E-02 & 8.7914E-07 \\
%16  & 6.5052E+08 & 2.8284E+03 & 5.7477E+08 & 4.9821E+01 & 2.0370E-04 \\
%20  & 2.5701E+11 & 3.1623E+03 & 1.9657E+10 & 9.1814E+01 & 3.8922E-03 \\
% 4  & 1.3269e+01 & 1.4142e+03 & 1.3422e+04 & 7.4647e-13 & 7.4647e-13  & 5.5614e-17\\   
% 8  & 4.4285e+03 & 2.0000e+03 & 4.8236e+05 & 6.6015e-08 & 1.2379e-11  & 2.5663e-17\\     
%12  & 1.6705e+06 & 2.4495e+03 & 1.6734e+07 & 1.2773e-02 & 2.1164e-10 & 1.2647e-17\\
%16  & 6.5052e+08 & 2.8284e+03 & 5.7477e+08 & 9.1483e+01 &1.4017e-08 & 2.4387e-17 \\
%20  & 2.5701e+11 & 3.1623e+03 & 1.9657e+10 & 1.6745e+02 & 3.9920e-07 & 2.0308e-17\\
% \end{tabularx}
\end{tabular}
	}
	{\rule{\temptablewidth}{1pt}}
  \caption{The loss of biorthogonality of CGS-Biorth and MGS-Biorth algorithm.}
\label{tab:biorth_ill_conditioned}
%We do experiments in MATLAB, where $\varepsilon = 2.2204e-16$.
\end{table}

\begin{remark}[High performance implementation]
To achieve better efficiency and parallel scalability on high performance platform, such as distributed memory multiprocessors, 
one may develop a block version of MGS-Biorth just as the block Gram-Schmidt orthogonalization \cite{Stewart2008Block},
utilizing level 3 BLAS to implement matrix-matrix product. 
For brevity, we omit details here and refer the readers to \cite{Li2023GCGE,Stewart2008Block}.
\end{remark}

\subsection{The main algorithm}

Based on the direct sum decomposition of biorthogonal invariant subspaces and the stable and efficient implementation of biorthogonalization,
we design a subspace iteration with projection method for solving LREP \eqref{equ:xy_algebra} prescribed by Algorithm \ref{alg:biorth_gcg4d}, 
and provide a matrix-free interface for large-scale problem.
There are four critical issues to address:
\begin{itemize}
	\item[(1)] How to deal with zero eigenvalues when $K$ is a SPSD matrix.
	\item[(2)] How to construct a basis of an approximate invariant subspace of $H$.
	\item[(3)] How to deflate out converged eigenpairs such that they \textsl{do not} participate in subsequent iterations.
	\item[(4)] How to maintain stability, efficiency, and parallel scalability when a large number of eigenpairs are requested. 
\end{itemize}

Briefly speaking, 
\comm{according to Lemma \ref{lem:zero},
we can first compute the generalized nullspace} $\mathcal{V}_0$ of $H$ in Theorem \ref{thm:dsd_biorth_is} when $H$ has zero eigenvalues,
and then construct a sequence of approximate invariant subspaces that lie in $\mathcal{V}_{1}\oplus \cdots \oplus \mathcal{V}_{n-r}$.
The bases of approximate invariant subspaces are constructed using Algorithm 
%\ref{alg:block_biorth} 
\ref{alg:biorth} 
so as to guarantee biorthogonality.
Once $\mathcal{V}_{1}\oplus \cdots \oplus \mathcal{V}_{{\ell}}$ are available for some ${\ell}\geq1$,
the subsequent approximate invariant subspaces shall be constructed in the biorthogonal complement, i.e.,
$\mathcal{V}_{{\ell}+1}\oplus \cdots \oplus \mathcal{V}_{n-r}$. 
Then the converged eigenpairs will not participate in subsequent iterations naturally.
When a large number of eigenpairs are required, 
we propose an effective moving mechanism to reduce memory consumption and improve the efficiency in Section \ref{sec:moving_mechanism}.

When $H$ has zero eigenvalues, i.e., ${\rm rank}(K)=r<n$, 
a basis of $\mathcal{V}_0$ in \eqref{equ:decomp_invar_subsp} needs to be calculated first. 
As presented in \textbf{Remark} \ref{rem:V0}, we get $X_0$ and $Y_0$ by solving the symmetric eigenvalue problems $KX_0={\bf 0}$
and the SPD linear equations $MY_0=X_0$,
where $X_0,\ Y_0\in\mathbb{R}^{n\times r}$
and $X_0^{\top}Y_0 = I_r$.
% according to Theorem \ref{thm:dsd_biorth_is}. 
Thus, we obtain a basis of $\mathcal{V}_0$, 
% which consists of columns of 
\revise{which is formed by the columns of}
$\begin{bsmallmatrix} Y_0 & O \\ O & X_0 \end{bsmallmatrix}$.

To construct a subspace lying in $\mathcal{V}_{1}\oplus \cdots \oplus \mathcal{V}_{n-r}$,
we randomly generate two large matrices
$U=[X,P,W]$, $V=[Y,Q,Z]$ with $X,\ P,\ W,\ Y,\ Q,\ Z\in\mathbb{R}^{n\times n_{e}}$, assuming
that %$U,\ V$ are both full column rank and 
$U^\top V$ is invertible. 
Applying %equation \eqref{orth-deduction} and 
%a biorthogonalization procedure (Algorithm \ref{alg:block_biorth}) to $U$, $V$, we can get
a biorthogonalization procedure (Algorithm \ref{alg:biorth}) to $U$ and $V$, we can get
$X_0^{\top}V=Y_0^{\top}U=O$, and $U^{\top}V=I_d$ with $d = 3 \cdot n_{e}$,
from which we can construct \comm{a $2d$-dimensional search subspace} 
\begin{equation}
\label{equ:subspace}
\mathcal V = {\rm span}\left(\begin{bmatrix} V & O \\ O & U \end{bmatrix}\right). 
\end{equation}
The original LREP \eqref{equ:xy_algebra} is projected into the above subspace $\mathcal V$
based on an oblique projector 
$\Pi:=
\begin{bsmallmatrix} V & O \\ O & U \end{bsmallmatrix}
\begin{bsmallmatrix} U & O \\ O & V \end{bsmallmatrix}^{\top}$
\cite{Saad2011Numerical}
such that the following properities holds 
\begin{equation*}
	\Pi
	\begin{bmatrix} y \\ x \end{bmatrix}\in\mathcal{V}
~\mbox{ and }~
(I_{2n}-\Pi)\begin{bmatrix} y \\ x \end{bmatrix} \pb 
\mathcal{V},\quad
\forall~\begin{bmatrix}y\\x\end{bmatrix}\in \mathbb{R}^{2n}.
\end{equation*}
Therefore, we obtain a structure-preserving projection matrix $\hat{H}  \in \mathbb{R}^{2d \times 2d}$, defined as
\begin{equation}
	\label{equ:hat_H}
	\hat{H} = 	
\begin{bmatrix} U & O \\ O & V \end{bmatrix}^{\top}
H
\begin{bmatrix} V & O \\ O & U \end{bmatrix}
=\begin{bmatrix}
	O & U^{\top}KU \\ V^{\top}MV & O
\end{bmatrix} \in \mathbb R^{2d\times 2d}.
\end{equation}
A similar formula is also given by Bai et al. in \cite{Bai2012Minimization}.
It is easy to prove that $U^{\top}KU$ and $V^{\top}MV$ are both SPD. 
Therefore, there are no zero eigenvalues in $\hat H$. This means that the zero eigenvalues are excluded with biorthogonalization.
In fact, if there exists a \textsl{non-zero} vector ${x}\in\mathbb{R}^{d}$ satisfying ${x}^{\top} U^{\top}KUx=0$, $Ux$ is a linear combination of the columns of $X_0$. 
Then there exists a vector $c\in\mathbb{R}^{r}$ such that $U{x} = X_0c$. 
% Multiplying $V^\top$ from the left on both sides 
\revise{Multiplying both sides on the left by $V^\top$} immediately leads to a contradiction ${x} = V^{\top}U{x} = V^{\top}X_0c= {\bf 0}$.
The positivity of $V^\top M V$ can also be proved in a similar way.

For the small-scale LREP \eqref{equ:small_eigen}, we adopt the structure-preserving method proposed by 
Shao and Benner et al. \cite{Benner2022efficientbse,Shao2016Structure} to
compute the first $n_{e}$ smallest positive eigenvalues \revise{and their corresponding eigenvectors, which satisfy}
% and denote the solutions as follows
\begin{equation}\label{equ:small_eigen}
\hat{H}
\begin{bmatrix}
	\hat{Y} \\ \hat{X}
\end{bmatrix}=
\begin{bmatrix}
	\hat{Y} \\ \hat{X}
\end{bmatrix}\Lambda,
\end{equation}
where $\Lambda = {\rm diag}\{\hat{\lambda}_1,\ldots,\hat{\lambda}_{n_{e}}\}$ and $\hat X,\ \hat Y\in \mathbb R^{d\times {n_{e}}}$ satisfy the biorthogonality condition $\hat{X}^{\top}\hat{Y}=I_{n_{e}}$. 
We can use $\hat{X}$ and $\hat{Y}$ to construct approximate eigenvectors, denoted as $\widetilde{X},\ \widetilde{Y}\in \mathbb R^{n\times {n_{e}}}$, 
to the original LREP \eqref{equ:xy_algebra}, as 
\begin{equation}
\label{eigenVector-pullback}
\widetilde{Y}=V\hat{Y},\quad \widetilde{X}=U\hat{X}.
\end{equation}
It is easy to check that the biorthogonality is kept in $\widetilde{X}$ and $\widetilde{Y}$, i.e.,
\begin{equation*}
	\widetilde{X}^{\top}\widetilde{Y} = \hat{X}^{\top}U^{\top}V\hat{Y} = \hat{X}^{\top}\hat{Y} = I_{n_{e}},
\end{equation*}
and each column of residual 
\begin{equation*}
	\label{equ:res_eig}
	\begin{bmatrix} R_1 \\ R_2\end{bmatrix}:=
	H
	\begin{bmatrix} \widetilde{Y}\\ \widetilde{X} \end{bmatrix}
	-\begin{bmatrix} \widetilde{Y}\\ \widetilde{X} \end{bmatrix}\Lambda
	=
	H
	\begin{bmatrix} V & O \\ O & U \end{bmatrix}
	\begin{bmatrix} \hat{Y} \\ \hat{X} \end{bmatrix}
-
	\begin{bmatrix} V & O \\ O & U \end{bmatrix}
	\begin{bmatrix} \hat{Y} \\ \hat{X} \end{bmatrix} \Lambda
\end{equation*} 
is biorthogonal to the %$2d$-dimensional 
subspace $\mathcal{V}$ defined by \eqref{equ:subspace},
i.e., $R_1^\top U= R_2^\top V = O$.

%\

\vspace{0.2cm}
\begin{remark}%[Trace minimization principle]
	As is already proved by  
	Bai et al. \cite{Bai2012Minimization}, one can seek reliable approximations to the first few smallest positive eigenvalues, i.e.,
	$\lambda_i, \, 1 \leq i \leq {n_{e}}$, simultaneously 
	by minimizing ${\rm trace}(\overline{X}^{\top}K\overline{X} + \overline{Y}^{\top}M\overline{Y})$
	subject to biorthogonal constraint $\overline{X}^{\top}\overline{Y}=I_{n_{e}}$ 
	with ${\rm span}(\overline{X}) \subset {\rm span}(U)$ and ${\rm span}(\overline{Y}) \subset {\rm span}(V)$.  That is to say, 
 \begin{equation}
		\label{equ:min_principle}
		\min_{\scriptstyle 
			\substack{
				{\rm span}(\overline{X})\subset {\rm span}(U)\\
				{\rm span}(\overline{Y})\subset {\rm span}(V)\\
			\overline{X}^{\top}\overline{Y}=I_{n_{e}}} 
		}
		\frac{1}{2}
		{\rm trace}\big(\overline{X}^{\top}K\overline{X} + \overline{Y}^{\top}M\overline{Y}\big).
	\end{equation}
It is easy to prove that the numerical solutions obtained by our method in Algorithm \ref{alg:biorth_gcg4d}, 
denoted by $\widetilde{X}$ and $\widetilde{Y}$,  indeed minimize the target trace function \eqref{equ:min_principle}.
\end{remark}

%\

If \comm{the norm of residual $\begin{bsmallmatrix}R_1 \\ R_2 \end{bsmallmatrix}$ is smaller than a pre-defined tolerance}, the iteration terminates and the first $n_{e}$ smallest positive approximate eigenvalues converge. In this case, $\mathcal V$ is indeed an invariant subspace of $H$. Otherwise,  %$U$ and $V$ shall be updated to continue.
setting
\begin{equation}
	\label{equ:hat_pq}
\hat{P} =(I_{d}-\hat{X}\hat{Y}^{\top})
(\hat{X}-\begin{bsmallmatrix} I_{n_{e}} \\ O \\ O\end{bsmallmatrix}),\quad
\hat{Q} =(I_{d}-\hat{Y}\hat{X}^{\top})
(\hat{Y}-\begin{bsmallmatrix} I_{n_{e}} \\ O \\ O \end{bsmallmatrix})\in\mathbb{R}^{d\times{n_{e}}}, 
\end{equation}
we do the small-scale biorthogonalization for 
$\hat{P}$ and $\hat{Q}$. 
New searching directions are updated as  $\widetilde{P}=U\hat{P}$, $\widetilde{Q}= V\hat{Q}$.
Using $[\hat{X},\hat{P}]^{\top}[\hat{Y},\hat{Q}]=I_{2\cdot{n_{e}}}$, it is easy to check that $[X_0, \widetilde{X}]^{\top}\widetilde{Q}
=[Y_0, \widetilde{Y}]^{\top}\widetilde{P} = O$ and $\widetilde{P}^{\top}\widetilde{Q}=I_{n_{e}}$.
The large-scale biorthogonalization for $\widetilde{P}$ and $\widetilde{Q}$ is then transformed into a small-scale biorthogonalization for $\hat{P}$ and $\hat{Q}$, 
and it definitely leads to a remarkable efficiency boost. %, especially when ${\tt nev} \ll n$.
In addition, we observe ${\rm span}([\widetilde{X}, X]) = {\rm span}([\widetilde{X}, \widetilde{P}])$ and ${\rm span}([\widetilde{Y}, Y]) = {\rm span}([\widetilde{Y}, \widetilde{Q}])$, 
which means that the updated $\widetilde{P}$ and $\widetilde{Q}$
contain information of eigenvectors from previous iteration steps.
%\comm{
%As shown in equation \eqref{equ:hat_pq},
%we use the classical Gram-Schmidt biorthogonalization to construct $\hat{P}$ and $\hat{Q}$,
%which is might cause numerical instability but that's not the case.
%This is because the condition number of the matrix $\hat{X}^{\top}\hat{Y}$
%is close to one. 
%As presented in \textbf{Remark} \ref{rem:biorth_ill_conditioned},
%some column vectors may be discarded
%during the biorthogonalization of $\hat{P}$ and $\hat{Q}$.
%This situation does not materially affect the stability of the procedure.
%}

\iffalse
\vspace{0.2cm}
\begin{remark}
When $\hat{P}^{\top}\hat{Q}$ is ill-conditioned in equation \eqref{equ:hat_pq},
some column vectors may be deleted from $\hat{P}$ and $\hat{Q}$ in the biorthogonalization procedure,
as presented in \textbf{Remark} \ref{rem:biorth_ill_conditioned}.
Then, the number of columns of $\widetilde{P}$ and $\widetilde{Q}$ is less than ${n_{e}}$.
This situation does not materially affect the stability.
\end{remark}
\fi

\comm{To update the search subspace $\mathcal{V}$}, 
we propose to update $w_i,\ z_i\in \mathbb R^{n}$, 
which is the $i$-th column of matrices $W$ and $Z$ respectively,
by solving the following linear system
\begin{equation}
\label{equ:residual_ZW}
(H-\hat{\lambda}_i I_{2n}) \begin{bmatrix} z_{i} \\ w_{i} \end{bmatrix} = -(H-\hat{\lambda}_i I_{2n})
\begin{bmatrix} \widetilde{y}_i \\ \widetilde{x}_i \end{bmatrix},
\end{equation}
which is equivalent to
\begin{equation}
\label{equ:residual_ZW_2}
\begin{bmatrix} M & -\hat{\lambda}_iI_n \\ -\hat{\lambda}_iI_n & K \end{bmatrix}
\begin{bmatrix} z_{i} \\ w_{i}\end{bmatrix}
= \begin{bmatrix} \hat{\lambda}_i\widetilde{x}_i-M\widetilde{y}_i \\  \hat{\lambda}_i\widetilde{y}_i-K\widetilde{x}_i \end{bmatrix}, 
\end{equation}
where 
\comm{$\hat{\lambda}_i$ is obtained by solving the small-scale LREP \eqref{equ:small_eigen}},
and 
$\widetilde{x}_i$ and $\widetilde{y}_i$ is the $i$-th column of $\widetilde{X}$ and $\widetilde{Y}$, respectively.
Bear in mind that we just need to add new search directions, so there is no need to solve the above linear system with very high accuracy.
In practice, we compute $w_i,z_i$ with block Gauss-Seidel iteration for \eqref{equ:residual_ZW_2} by solving
a series of symmetric linear equations 
\begin{equation}\label{equ:block_GS}
MZ^{(j)} = W^{(j-1)}\Lambda
+(\widetilde{X}\Lambda - M\widetilde{Y}), \quad
KW^{(j)} = Z^{(j)}\Lambda
+(\widetilde{Y}\Lambda - K\widetilde{X}),
\end{equation}
for $j=1,\ldots,{\tt ngs}$,
where $W^{(j)},\ Z^{(j)}\in\mathbb{R}^{n\times{n_{e}}}$ with zero matrix as initial guess, i.e., $W^{(0)} = O$, 
and ${\tt ngs}$ denotes the number of Gauss-Seidel iteration steps. 
\revise{The linear system in equation~\eqref{equ:block_GS} is solved using the conjugate gradient method with a zero initial guess, 
a stopping tolerance of $10^{-2}$, and a maximum of $20$ iterations.
}

%$W^{(0)}=O$.
\comm{
%Since the condition number of 
%$[X_0,\widetilde{X},\widetilde{P}]^{\top}[Y_0,\widetilde{Y},\widetilde{Q}]$
%is close to one,
%we can use the classical Gram-Schmidt biorthogonalization to construct
Then, we carry out biorthogonalization to $W$/$Z$ with respect to previous search directions as follows 
\begin{align}
\widetilde{W}& =(I_{n} - [X_0,\widetilde{X},\widetilde{P}] [Y_0,\widetilde{Y},\widetilde{Q}]^{\top})W^{(\tt ngs)},
	\label{equ:update_W}\\
\widetilde{Z}& =(I_{n} - [Y_0,\widetilde{Y},\widetilde{Q}] [X_0,\widetilde{X},\widetilde{P}]^{\top})Z^{(\tt ngs)},
	\label{equ:update_Z}
\end{align}
\comm{
	which costs $16n\cdot{n_{e}}^2$ flops and is accelerated using the level 3 BLAS routines,
}
followed by a large-scale biorthogonalization for 
%$\widetilde{W}$ and $\widetilde{Z}$ by Algorithm \ref{alg:block_biorth},
$\widetilde{W}$ and $\widetilde{Z}$ using Algorithm \ref{alg:biorth},
}
\comm{which requires another $4n\cdot{n_{e}}^2$ flops.}
Finally, 
we update 
$U=[X,P,W]$ 
and 
$V=[Y,Q,Z]$ 
by 
$[\widetilde{X},\widetilde{P},\widetilde{W}]$
and
$[\widetilde{Y},\widetilde{Q},\widetilde{Z}]$,
respectively.
%The next iteration starts in the loop.

In summary, the BOSP algorithm, 
presented in Algorithm \ref{alg:biorth_gcg4d},
generates two triple blocks 
$U = [X,P,W]$ and $V = [Y,Q,Z]$ with $U^{\top}V=I$,
where $\begin{bsmallmatrix}Y\\X \end{bsmallmatrix}$
stores the current approximations of eigenvectors,
$\begin{bsmallmatrix}Q\\P \end{bsmallmatrix}$ saves the information of eigenvectors at the previous iteration step,
and $\begin{bsmallmatrix}Z\\W \end{bsmallmatrix}$ saves 
the approximate Newton search direction
at the point $\begin{bsmallmatrix}Y\\X \end{bsmallmatrix}$.
Similar ideas were applied to symmetric eigenvalue problems 
\cite{Li2023GCGE,Li2020parallel,Zhang2020generalized}.

\begin{algorithm}[!htb]
\setstretch{1.2}
\caption{BOSP algorithm} \label{alg:biorth_gcg4d}
\KwIn{Given the number of requested eigenpairs $n_{e}\in \mathbb{N}$, convergence tolerance $\varepsilon_{\tt tol}$ and 
the generalized nullspace $\mathcal{V}_0 = {\rm span}(\begin{bsmallmatrix} Y_0 & O \\ O & X_0 \end{bsmallmatrix})$,
where $KX_0=O$, $MY_0=X_0$ and $X_0^{\top}Y_0=I_{r}$ with ${\rm dim}(\mathcal{N}(K)) = r>0$ and $X_0, Y_0\in\mathbb{R}^{n\times r}$.} 
% \\[0.3em]
%\vspace{0.3cm}
\KwOut{The first $n_{e}$ smallest positive eigenvalues and their corresponding eigenvectors $\left\{\Lambda; \begin{bsmallmatrix} {Y} \\ {X}\end{bsmallmatrix}\right\}$,
	satisfying $KX=Y\Lambda$, $MY=X\Lambda$ and $X^{\top}Y=I_{n_{e}}$, 
	where $X,\ Y\in\mathbb{R}^{n\times {n_{e}}}$ and $\Lambda = {\rm diag}\{\hat{\lambda}_1,\cdots,\hat{\lambda}_{n_{e}}\}$ with 
$0< \hat{\lambda}_1 \leq \cdots\leq \hat{\lambda}_{n_{e}}$.} 
% \\[0.3em]
\vspace{0.2cm}
\textbf{Step 1:} Generate $6\cdot{n_{e}}$ column vectors randomly to construct $U=[X,P,W]$, $V=[Y,Q,Z] \in \mathbb{R}^{n\times d}$, 
where $X,\ P,\ W,\ Y,\ Q,\ Z \in \mathbb{R}^{n\times {n_{e}}}$ and $d = 3\cdot {n_{e}}$. \\%[0.5em]

\textbf{Step 2:} Update $U = (I_n - X_0 Y_0^{\top})U$, $V = (I_n - Y_0 X_0^{\top})V$ and biorthogonalize $U$ and $V$ by Algorithm 
\ref{alg:biorth}. \\%[0.5em]		
%\ref{alg:block_biorth}. \\%[0.5em]		

Set $\varepsilon_i =1, \mbox{ for } i=1,\ldots, {n_{e}}$. \\%[0.3em]

\While{not converged (i.e., $\max\left\{\varepsilon_{i}\right\}_{i=1}^{n_{e}} >  \varepsilon_{\tt tol}$)}{
   \vspace{0.2em}
\textbf{Step 3:} Generate matrix $\hat{H}$, compute the small-scale LREP \eqref{equ:small_eigen}
to obtain $\left\{\Lambda;\begin{bsmallmatrix}\hat{Y} \\ \hat{X}\end{bsmallmatrix}\right\}$, and update $\widetilde{X}=U\hat{X}$ and $\widetilde{Y} = V\hat{Y}$.\\%[0.3em]
\textbf{Step 4:} Compute the residuals $\varepsilon_{i}$ = $\frac{\Vert H\xi_i - \hat{\lambda}_i \xi_i \Vert_2}{(1+ \hat{\lambda}_i) \Vert  \xi_i \Vert_2}$ where
	$\xi_i$ is the $i$-th column of $\begin{bsmallmatrix} \widetilde{Y} \\ \widetilde{X}\end{bsmallmatrix}$. 
	If converged, update $X=\widetilde{X}$, $Y=\widetilde{Y}$ and the while-loop stops.\\%[0.3em]
\textbf{Step 5:} Update $\hat{P}$ and $\hat{Q}$ by \eqref{equ:hat_pq} and biorthogonalize $\hat{P}$ and $\hat{Q}$. Update $\widetilde{P}=U\hat{P}$, $\widetilde{Q}=V\hat{Q}$.\\%[0.3em]
\textbf{Step 6:} Update new search directions $[\widetilde{W},\widetilde{Z}]$ by %\eqref{equ:residual_ZW}
\eqref{equ:block_GS} 
and biorthogonalize them using \eqref{equ:update_W} and \eqref{equ:update_Z}.\\%[0.3em]
\textbf{Step 7:} Update $U = [\widetilde{X},\widetilde{P},\widetilde{W}]$ and $V =[\widetilde{Y},\widetilde{Q},\widetilde{Z}]$.%\\[0.3em]
	}
\end{algorithm}

It is important to notice that 
we do biorthogonalization for $U$ and $V$,
which is transformed into a small-scale biorthogonalization for $\hat{P}$ and $\hat{Q}$ 
and a large-scale biorthogonalization for $\widetilde{W}$ and $\widetilde{Z}$.
\comm{
Roughly specking, the computational complexity decreases from 
$36\,n\cdot{n_{e}}^2$ to $20\,n\cdot{n_{e}}^2$.}
And preserving the biorthogonality of $U$ and $V$
is the critical factor of our algorithm 
to guarantee the numerical stability.
When high accuracy of eigenpairs is required, 
the columns of $[\widetilde{X},\widetilde{P},\widetilde{W}]$
or $[\widetilde{Y},\widetilde{Q},\widetilde{Z}]$ are nearly linear dependent
without the biorthogonalization for $\widetilde{W}$ and $\widetilde{Z}$.
In the next iteration step, 
the matrix $\hat{H}$ should then be defined as
\begin{equation*}
	\begin{bmatrix}
		O & (US_1^{-1})^{\top}K(US_1^{-1}) \\ (VS_2^{-1})^{\top}M(VS_2^{-1}) & O
	\end{bmatrix},
\end{equation*}
as done in LOBP4dCG method \cite{Bai2012Minimization,Bai2013Minimization,Bai2016Linear},
where $U^{\top}V=S=S_1^{\top}S_2$.
But under the circumstances stated above, $S$ is nearly singular and 
the operation of inverting $S_1$ and $S_2$
is numerically unstable.

%\
\vspace{0.2cm}
\begin{remark}[Generalized LREP]
For a class of generalized LREP
\begin{equation}
   \label{equ:g_LREP}
\begin{bmatrix}	O & K \\ M & O \end{bmatrix}
\begin{bmatrix} y \\ x \end{bmatrix}
=\lambda 
\begin{bmatrix} B & O \\ O & B \end{bmatrix}
\begin{bmatrix} y \\ x \end{bmatrix},
\end{equation}
we can induce $B$-inner product 
$\langle  x,y\rangle_B = x^{\top}B y$,
where $B \in \mathbb{R}^{n\times n}$ is a SPD matrix and $x,\ y\in\mathbb{R}^n$.
The biorthogonalization algorithm presented in Section \ref{sec:mgs}  %are defined based on this norm.
can also be implemented in 
$B$-inner product instead of standard Euclidean inner product.
Then it is straightforward to extend Algorithm \ref{alg:biorth_gcg4d} 
to solve the generalized LREP \eqref{equ:g_LREP}.
\end{remark}

\subsection{The deflation mechanism}
\label{sec:moving_mechanism}

As stated in Theorem \ref{thm:biorth_compl}, 
the new search subspace $\mathcal V$, defined by \eqref{equ:subspace}, can be constructed within the biorthogonal complement of converged subspace. 
Consequently, the deflation mechanism is naturally realized without introducing any artificial parameters and converged eigenpairs \textsl{will not} participate 
in the subsequent computations. This form of deflation is precisely analogous to that used for symmetric eigenvalue problems \cite{Hernandez2005SLEPc,Li2023GCGE}. 
Most remarkably, there is no requirement for \textsl{a priori} spectral distribution of $H$.

In practice, when the first $\ell$ smallest positive eigenvalues have converged, we obtain the invariant subspace
$\mathcal V_0 \oplus \mathcal V_1 \oplus \cdots \oplus \mathcal V_\ell$ at hand, 
and \revise{the column vectors of $U$ and $V$ should be updated} in the biorthogonal complementary subspace,
namely,
\begin{equation*}
\mathcal V ~\subset~ \mathcal{V}_{\ell+1}\oplus \cdots \oplus \mathcal{V}_{n-r} ~ ~ \pb~ ~\mathcal V_0 \oplus \mathcal V_1 \oplus \cdots \oplus \mathcal V_\ell.
 \end{equation*}
The most CPU-intensive task is the small-scale eigenvalue problem (\textbf{Step 3} in Algorithm \ref{alg:biorth_gcg4d}), 
which is manifested by Table \ref{tab:nev5000sih4} in next section,
and it involves dense matrix generation \eqref{equ:hat_H}, eigenvalue computation \eqref{equ:small_eigen} 
and eigenvectors construction \eqref{eigenVector-pullback}.
To illustrate the complexity, we measure the computation costs in terms of matrix-vector multiplications, denoted as \textbf{mv}, 
vector inner product (denoted by \textbf{vp}) and floating point operations (flops).
The costs are detailed as follows:
\begin{itemize}
	\item[(1)]The matrix generation of $\widehat{H}$ in \eqref{equ:hat_H} needs $2d$ times \textbf{mv} and $(d^2\!+\!d)$ times \textbf{vp},
\item[(2)]The eigenvalue computation of \eqref{equ:small_eigen} is done using two Cholesky factorizations and one singular value decomposition 
	within $\mathcal{O}(d^3)$ flops \cite{Shao2016Structure},
\item[(3)]The eigenvectors construction of $\widetilde{X}$ and $\widetilde{Y}$ in \eqref{eigenVector-pullback} requires $\mathcal{O}(nd^2)$ flops.
\end{itemize}
To sum up, the total complexity of \textbf{Step 3} amounts to 
\begin{equation}\label{cmpl-lrep}
(2d)\cdot\textbf{mv} + (d^2\!+\!d)\cdot\textbf{vp}
+\mathcal{O}(d^3+nd^2).
\end{equation}

As mentioned earlier, the size of small-scale matrix $\widehat{H}$ is three times the numbe of $n_{e}$, i.e., $ d = 3\cdot{n_{e}}$.
% It is noteworthy to point out 
It is worth noting that flops cost for one \textbf{mv} is at least $\mathcal{O}(n)$, which corresponds to sparse matrices 
that might arise from finite difference/element/volume spatial discretization, and flops for one \textbf{vp} is normally $\mathcal{O}(n)$.
Therefore, the leading order of the above complexity \eqref{cmpl-lrep} is at least 
$
%\mathcal{O} ~(~ ({\tt nev})^3 +  ({\tt nev})^2~ n + ({\tt nev}) ~n ~).
\mathcal{O} (n\cdot n_{e}^2 + n_{e}^3).
$
In real applications, there are often cases where the matrix dimension is very large, i.e., $n\gg 1$
(for example, matrix resulted from discretization of $3$-dimension PDE eigenvalue problem), 
and/or the number of required eigenpairs is very large, i.e., $n_{e} \gg 1$ (for example, the Bethe-Salpeter equation).
An even more challenging case requires computing very large number of eigenpairs for a large-scale system ($n_{e} \gg 1~\& ~ n \gg 1$).
In such cases, the computation of small-scale problem bottlenecks the efficiency performance even  on distributed memory multiprocessors,
and, at the same time, it requires a prohibitively huge chunk of memory storage.

\begin{figure}[!htbp]
\centering
\subfigure[Computing the eigenpairs in batches]
{ \includegraphics[width=0.46\textwidth]{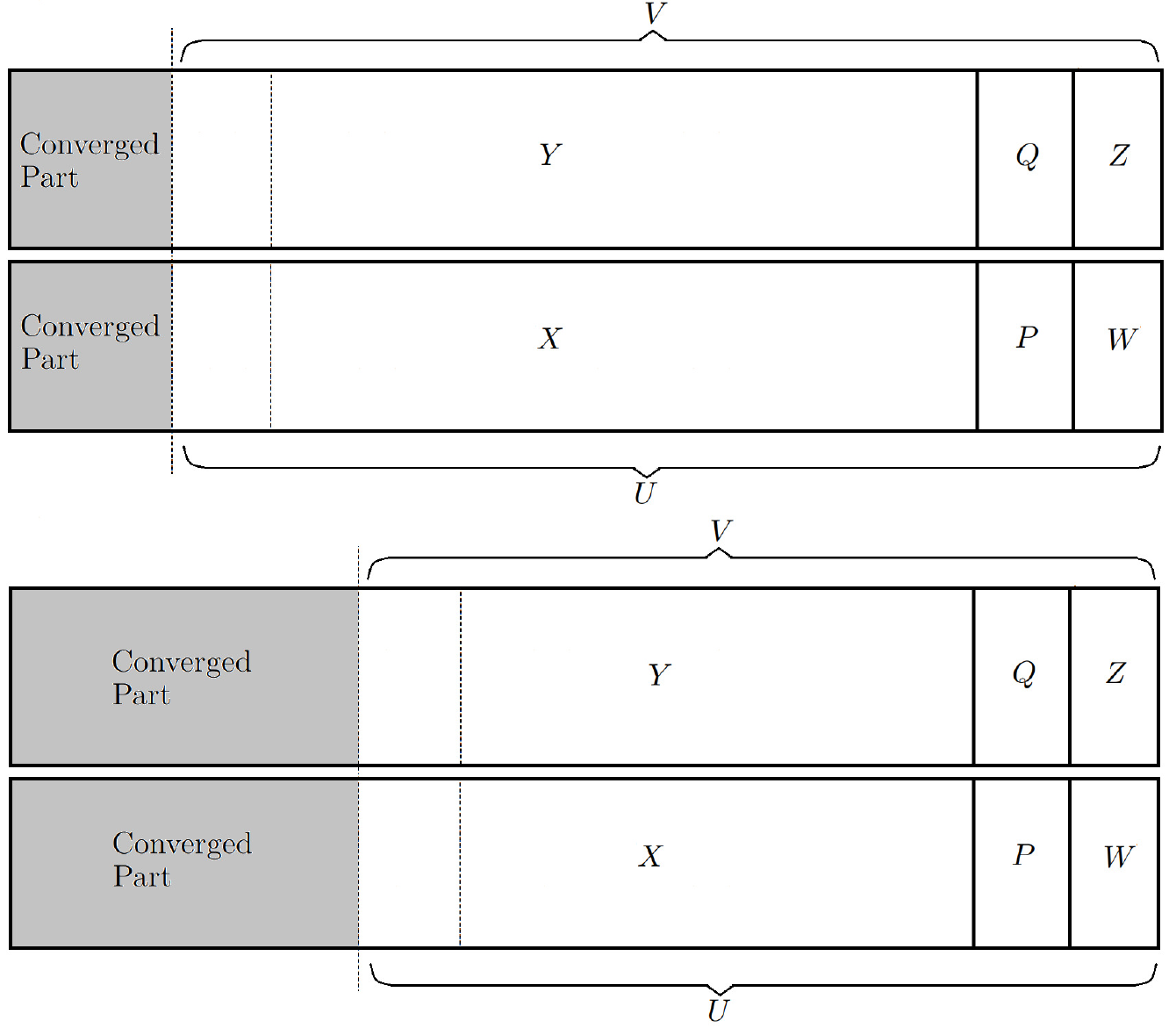} }
\subfigure[Moving $U$ and $V$ when $2\cdot n_b$ eigenpairs converged]
{ \includegraphics[width=0.46\textwidth]{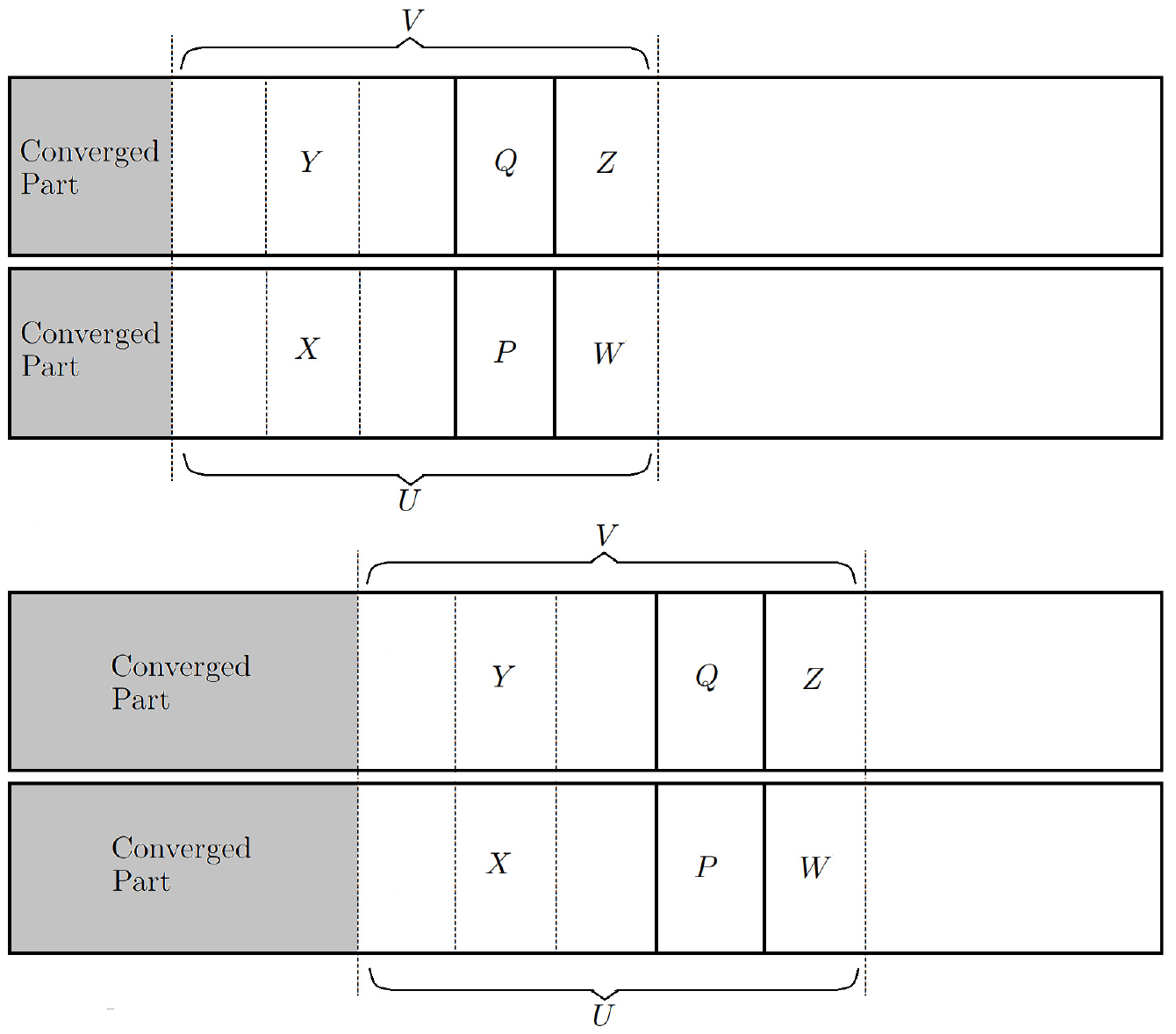} }
\caption{The batching and moving mechanism.}
\label{fig:xpw}
\end{figure}
To improve efficiency, for a fixed number of degrees of freedom $n$, we choose to decrease the dimension of search subspace $\mathcal V$, or equivalently the size 
of small-scale eigenvalue problem \eqref{equ:small_eigen}, by computing eigenpairs batch by batch.
This batching strategy is analogous to that for the symmetric eigenvalue problem 
\cite{Hernandez2005SLEPc,Li2023GCGE}.
To be precise, only the first $n_b$ unconverged eigenpairs are utilized to update the search vectors
$\widetilde{P},\ \widetilde{W},\ \widetilde{Q},\ \widetilde{Z}$ in {\bf Step 5-6}.
The batch size $n_b$ is set to the default value
$n_b=\min\{n_{e}/5, 150\}$.
%when ${\tt nev} \geq 5$. If ${\tt nev} < 5$, we set $n_b ={\tt nev}$. 
A schematic illustration is shown in Figure \ref{fig:xpw} (a), where the number of columns of $P,\ Q,\ W,\ Z$ is all set as $n_b$.
Once the first $n_b$ eigenpairs have converged, we shall continue to compute the next $n_b$ pairs.
The dimension of search subspace $\mathcal{V}$ is reduced from $6\cdot n_{e}$ to $ 2\cdot(n_{e}+2\cdot n_b)$ at the initial batch 
and diminishes successively to $6\cdot n_b$ as an arithmetic progression with a common difference $2\cdot n_b$, that is,
\begin{align*}
   6\cdot n_{e} \quad &\longrightarrow& %&\searrow&
	 \boxed{\ 2\cdot(n_{e}+2\cdot n_b) \ \searrow \ 2\cdot(n_{e}+n_b) \ \searrow \ \cdots \ \searrow \ 6\cdot n_b \ }\\%[3mm]
\mbox{Algorithm \ref{alg:biorth_gcg4d}} & &\mbox{batching mechanism}\qquad\qquad\qquad\qquad
\end{align*}

To accelerate the computation efficiency further, we introduce a much more aggressive \textsl{moving} mechanism
to reduce the dimension of search subspace to a constant that is {\bf independent} of $n_{e}$.
The number of columns of $X$ and $Y$ \revise{is fixed at} $s\cdot n_b$ for some positive integer $s\geq2$.
The eigenpairs are computed in batches with $s=3$ as default value.
Once $2\cdot n_b$ eigenpairs have converged,  we integrate $[P,W]$ and $[Q,Z]$ into $X$ and $Y$, 
and update $[P,W]$ and $[Q,Z]$ following {\bf Step 5-6} of Algorithm \ref{alg:biorth_gcg4d}.
To sum up, the dimension of search subspace is kept unchanged as $2(s+2)\cdot n_b$
%\begin{equation}
%	\label{equ:d_H_moving}
%        $2(s+2)\cdot n_b$
%\end{equation}
and the complexity \eqref{cmpl-lrep} becomes 
$
\mathcal{O}(n\cdot n_b^2 + n_b^3). 
$
Therefore, the memory costs are greatly reduced and the overall efficiency is improved to a great extent. 
A schematic illustration of the moving mechanism is depicted in Figure \ref{fig:xpw} (b) with $s=3$, from which 
we can see that the maximum number of columns of $U$ and $V$ is $5\cdot n_b$.
\comm{
Meanwhile, 
using the moving mechanism,
the computational complexity 
of biorthogonalization 
further decreases from 
$20n\cdot n_{e}^2$ flops
to
$8n\cdot n_{e}\cdot n_b+4n\cdot n_b^2$ flops
at each BOSP iteration.
}

\iffalse
A comparison of \textsl{batching} and \textsl{moving} mechanism in terms of search subspace dimension is shown 
in Figure \ref{fig:dimHatH}, from which 
\begin{figure}[!htbp]
\centering
\includegraphics[scale=0.4]{imag/dimHatH-eps-converted-to.pdf}
\caption{The dimension of search subspace $\mathcal{V}$ with $n_b=\min\{n_{e}/5, 150\}$.}
\label{fig:dimHatH}
\end{figure}
\fi
We can see clearly that 
the moving mechanism reduces the subspace dimension greatly, 
especially when $n_{e}$ is very large, thus allowing for the computation of a very large number of eigenpairs. 
The stability and effectiveness of the moving mechanism rely heavily on the stable implementation of direct sum decomposition 
of biorthogonal invariant subspaces of $H$. 
The decomposition is successfully realized via MGS-Biorth algorithm, as detailed in Section \ref{sec:mgs}.
Similar moving mechanism has been applied to symmetric eigenvalue problems \cite{Hernandez2005SLEPc,Li2023GCGE}.

\vspace{0.2cm}
\begin{remark}[Balance of convergence and efficiency]
In practice, a larger search subspace $\mathcal{V}$ leads to a faster convergence,
but costs much more time in each small-scale LREP \eqref{equ:small_eigen}.
That is to say, one has to strike a balance between convergence and efficiency 
by adjusting the batch size $n_b$. 
We shall investigate the influence of $n_b$ on numerical performance in the coming section.
\end{remark}

\section{Numerical results}\label{sec:numerical}

In this section, 
we present some numerical results to illustrate 
the stability, efficiency, and parallel scalability of BOSP algorithm. 
\comm{Given the fact that some eigenvalues might be very close to zero, 
	we set the convergence criterion as the normalized residual is below 
 the given tolerance $\varepsilon_{\tt tol}$, that is,}
%\vspace{-0.6cm}
\begin{equation}
	\label{equ:tol}
	\frac{\Vert H \xi - \lambda \xi \Vert_2}{(1+\lambda) \Vert  \xi \Vert_2} < \varepsilon_{\tt tol},
	\mbox{ where } \ \xi = \begin{bmatrix}y \\ x\end{bmatrix}\in\mathbb{R}^{2n}.
\end{equation}
The algorithms \revise{are} implemented in C and run on a 3.00G Hz Intel(R) Xeon(R) Gold 6248R CPU with a 36 MB cache in Ubuntu GNU/Linux, 
unless otherwise stated. The OpenMP support is enabled with a default $8$ threads.

\subsection{Accuracy confirmation}
To illustrate the accuracy, 
we choose to present the absolute and relative errors of the first few smallest positive eigenvalues and their associated eigenvectors.
The exact and numerical eigenvalues are denoted as $\lambda_{\ell}^{\tt e}$ and $\lambda_{\ell}$, respectively. And 
the exact and numerical eigenvectors are 
%$z_{\ell}^{\tt e} = [{y_{\ell}^{\tt e}}^\top, {x_{\ell}^{\tt e}}^\top]^\top$ and 
%$z_{\ell} = [y_{\ell}^\top, x_{\ell}^\top]^\top$, respectively.
$\xi_{\ell}^{\tt e} = \begin{bsmallmatrix}{y_{\ell}^{\tt e}}\\ {x_{\ell}^{\tt e}}\end{bsmallmatrix}$ and 
$\xi_{\ell} = \begin{bsmallmatrix}y_{\ell} \\ x_{\ell}\end{bsmallmatrix}$, respectively.
The errors are defined as 
$\varepsilon_{\ell} = \frac{|\lambda_{\ell}^{\tt e}-\lambda_{\ell}|}{\lambda_{\ell}^{\tt e}}$ and
$\eta_{\ell} = \Vert\frac{\xi^{\tt e}_{\ell}}{\Vert \xi^{\tt e }_{\ell}\Vert_2}-\frac{\xi_{\ell}}{\Vert \xi_{\ell}\Vert_2}\Vert_2$.
%\begin{equation*}
%\varepsilon_{\ell} = \frac{|\lambda_{\ell}^{\tt e}-\lambda_{\ell}|}{\lambda_{\ell}^{\tt e}},\qquad
%\eta_{\ell} = \Vert\frac{\xi^{\tt e}_{\ell}}{\Vert \xi^{\tt e }_{\ell}\Vert_2}-\frac{\xi_{\ell}}{\Vert \xi_{\ell}\Vert_2}\Vert_2.
%\end{equation*}
For convenience,
we define a matrix
\begin{equation*}
  \label{equ:mat_T}
  T(s) = 
  \begin{bsmallmatrix}
    2 & -1 & & & s \\
    -1 & 2 & -1 & &  \\
    & \ddots & \ddots & \ddots \\
    & & \ddots & \ddots & -1 \\
    s& & & -1& 2 \\
	\end{bsmallmatrix}_{n\times n}.
\end{equation*}

%========================
% COMPARISONS WITH ARPACK
%========================

%\vspace{0.2cm}
\begin{example}[Accuracy for SPD matrix case]\label{exam:acc_spd}
In this example, we consider a commonly-used SPD matrix $K$, resulting from finite difference discretization 
of the Laplacian operator with homogeneous Dirichlet boundary conditions in one-dimension space, 
and choose the second SPD matrix to match $K$. %, i.e., $M = K$. 
To be exact, 
$K = M = T(0)$.
\end{example}

It is easy to check that the exact positive eigenvalues of $H$ in LREP \eqref{equ:xy_algebra} are 
$\lambda_{\ell}^{\tt e} = 4 \sin^{2}(\frac{\pi {\ell}}{2(n+1)})$, and the 
associated eigenvectors are given explicitly as follows
\begin{equation*}
	x_{\ell}^{\tt e} = y_{\ell}^{\tt e} = [\sin(\frac{{\ell}\pi}{n+1}),\sin(\frac{2{\ell}\pi}{n+1}),\cdots,\sin(\frac{n{\ell}\pi}{n+1})]^{\top},\qquad
\mbox{for }{\ell} = 1,\ldots,n.
\end{equation*}
In Table \ref{tab:rel_err_spd2},
we present the absolute and relative errors of the first ten smallest positive eigenvalues
and their eigenvectors for $n=1000$  with $n_{e} = n_b = 10$ and $\varepsilon_{\tt tol} = 10^{-10}$, 
from which 
one can observe an almost machine precision accuracy of our algorithm.
\begin{table}[!htbp]
	\centering
	% {\rule{\temptablewidth}{1pt}}
%   \setlength{\tabcolsep}{1pt} 
  \resizebox{\linewidth}{!}{
	% \begin{tabularx}{\temptablewidth}{@{\extracolsep{\fill}}c|c|c|ccc}
		\begin{tabular}{c|c|c|ccc}
			\toprule
		$k$ & $\lambda_k^{\tt e}$ & $\lambda_k$ & $|\lambda_k^{\tt e}-\lambda_k|$ & $\varepsilon_k$ & $\eta_k$ \\
		\hline\rule{0pt}{12pt}
		\,1 & 9.849886676638340E-06 & 9.849886676638337E-06 & 3.38E-21 & 3.43E-16 & 7.05E-16 \\
		~2 & 3.939944968628582E-05 & 3.939944968628914E-05 & 3.32E-18 & 8.42E-14 & 2.50E-16 \\
		~3 & 8.864839796909544E-05 & 8.864839796903917E-05 & 5.62E-17 & 6.34E-13 & 2.34E-15 \\
		~4 & 1.575962464285077E-04 & 1.575962464285145E-04 & 6.85E-18 & 4.35E-14 & 9.75E-16 \\
		~5 & 2.462423159360287E-04 & 2.462423159360293E-04 & 5.96E-19 & 2.42E-15 & 4.30E-16 \\
		~6 & 3.545857333379193E-04 & 3.545857333379432E-04 & 2.39E-17 & 6.74E-14 & 9.04E-16 \\
		~7 & 4.826254314637962E-04 & 4.826254314637798E-04 & 1.63E-17 & 3.39E-14 & 1.64E-17 \\
		~8 & 6.303601491371425E-04 & 6.303601491372393E-04 & 9.68E-17 & 1.53E-13 & 2.34E-16 \\
		~9 & 7.977884311877310E-04 & 7.977884311877433E-04 & 1.22E-17 & 1.53E-14 & 1.21E-15 \\
		10 & 9.849086284659575E-04 & 9.849086284659519E-04 & 5.63E-18 & 5.72E-15 & 5.56E-16 \\
		\bottomrule
	\end{tabular}
  }
	\caption{Accuracy for SPD matrix case in \textbf{Example \ref{exam:acc_spd}} 
	for $n=1000$ with $n_{e}=n_b=10$ and $\varepsilon_{\tt tol} = 10^{-10}$.}
	\label{tab:rel_err_spd2}
\end{table}

%\vspace{0.2cm}
\begin{example}[Accuracy for SPSD matrix case]\label{exam:acc_sspd}
Here, we investigate a SPSD matrix $K$, where the generalized nullspace of $H$ 
in LREP \eqref{equ:xy_algebra} is nontrivial. Specifically, we choose the following matrices
$K = T(-1)$ and $M = T(0)$.
Matrix $K$ and $M$ are derived respectively by discretizing the Laplacian operator with 
periodic and homogeneous Dirichlet boundary condition on a uniform mesh grid in one-dimension space.
\end{example}

It is clear that 
we have
$Kx^0 ={\bf 0}$, $My^0 = x^0$,
and ${\rm rank}(K)=n-1$,
%\begin{equation}
%\label{equ:x_y_zero}
%\end{equation}
where
$x^0=[1, 1,\cdots,1]^{\top}$ and 
$y^0=[a_1, a_{2},\cdots,a_n]^{\top}$
with $a_{\ell} = a_{n-{\ell}+1} = {\ell}(n-{\ell}+1)/2$
for ${\ell} =1,\ldots,n/2$ when $n$ is even
and 
${\ell} =1,\ldots,(n+1)/2$ when $n$ is odd.
In general, we can get $x^0$ and $y^0$ numerically
by solving a symmetric eigenvalue problem 
and 
a symmetric linear equations problem % in \eqref{equ:x_y_zero}
using an appropriate algorithm \cite{Saad2003Iterative,Saad2011Numerical}.

The ``exact'' first ten smallest positive eigenvalues, served as benchmark solutions, are computed 
\comm{by function \textbf{eigs} in }Advanpix\footnote{https://www.advanpix.com/} toolbox using quadruple precision.
We carry out a detailed accuracy comparison with Matlab function \textbf{eigs}, and present 
the relative errors of eigenvalues in Table \ref{tab:rel_err_sspd}, 
from which we witness an accuracy superiority of our algorithm, even over the Matlab function.

\begin{table}[!htbp]
	\centering
	% {\rule{\temptablewidth}{1pt}}
  \resizebox{\linewidth}{!}{
	\begin{tabular}{c|c|c|c|c|c}
		\toprule
           $k$ &Advanpix   &   \textbf{eigs}  &  $\varepsilon_k$& our algorithm & $\varepsilon_k$\\
		\hline\rule{0pt}{12pt}
  \,1 & 3.943890108210E-05 & 3.943890108211E-05 & 3.85E-13  &3.943890108210E-05 &5.16E-15\\
	~2 & 6.154958719056E-05 & 6.154958719094E-05 & 6.33E-12  &6.154958719063E-05 &1.17E-12\\
	~3 & 1.577542931907E-04 & 1.577542931915E-04 & 5.01E-12  &1.577542931907E-04 &8.63E-14\\
	~4 & 1.994584196853E-04 & 1.994584196853E-04 & 4.51E-13  &1.994584196853E-04 &3.13E-13\\
	~5 & 3.549418750556E-04 & 3.549418750564E-04 & 2.50E-12  &3.549418750558E-04 &8.36E-13\\
	~6 & 4.161478616511E-04 & 4.161478616520E-04 & 2.31E-12  &4.161478616512E-04 &4.46E-13\\
	~7 & 6.309942290978E-04 & 6.309942290976E-04 & 1.92E-13  &6.309942290978E-04 &7.14E-14\\
	~8 & 7.116221744879E-04 & 7.116221744875E-04 & 5.99E-13  &7.116221744878E-04 &1.30E-13\\
	~9 & 9.859008227908E-04 & 9.859008227863E-04 & 4.53E-12  &9.859008227908E-04 &9.06E-15\\
	10 & 1.085870497647E-03 & 1.085870497651E-03 & 4.63E-12  &1.085870497646E-03 &5.26E-14\\
	% \end{tabularx}
	% {\rule{\temptablewidth}{1pt}}
	\bottomrule
\end{tabular}
}
\caption{Accuracy for SPSD matrix case in \textbf{Example \ref{exam:acc_sspd}} for $n=1000$ with $n_{e}=n_b=10$ and $\varepsilon_{\tt tol} = 10^{-10}$.}
	\label{tab:rel_err_sspd}
	\end{table}

\subsection{Efficiency performance}
To show the efficiency performance, in this subsection, we consider two pairs of matrices of $K$ and $M$ that are 
generated by turboTDDFT code in QUANTUM ESPRESSO for disodium (Na2) and silane (SiH4) \cite{Giannozzi2009QUANTUM}.
The size of $K$ is $1862$ and $5660$ for Na2 and SiH4 respectively.
Such two molecules are often used as benchmarks to assess various simulation models, functionals, and methods
\cite{Bai2013Minimization,Bai2016Linear}.

%\vspace{0.2cm}
\begin{example}[Small $n_{e}$ case]\label{exam:small_nev}
In this example, we investigate the convergence and efficiency when the number of eigenpairs, $n_{e}$, is small.
Here we compute the first ten smallest positive eigenvalues and their corresponding eigenvectors for Na2 and SiH4 problems
with the batch size $n_b=n_{e}$ for different tolerances, i.e., $\varepsilon_{\tt tol} = 10^{-6}$, $10^{-8}$, $10^{-10}$. 
\end{example} 

\revise{As shown in Table \ref{tab:iter_comparison}, compared with the widely used LOBP4dCG\footnote{\url{https://web.cs.ucdavis.edu/~bai/LReigsoftware/}},
%\cite{Bai2012Minimization,Bai2013Minimization,Bai2014Minimization,Bai2016Linear}, 
	BOSP converges faster in general,
	typically requiring only about half number of iterations. 
	Notably, for the SiH4 system at the highest precision, 
	LOBP4dCG failed to converge within $200$ iterations for even the first eigenpair (indicated by ``--''), while BOSP successfully obtained all desired eigenpairs in only $17$ iterations.
}
\begin{table}[htbp]
	\centering
	\renewcommand{\arraystretch}{1.0}
	\begin{tabular}{lccc}
	\toprule
	Problem & \raisebox{0.5ex}{$\varepsilon_{\tt tol}$}  & LOBP4dCG & BOSP \\
	\midrule
	\multirow{3}{*}{Na2} 
	& $10^{-6}$  & 9  & 6  \\
	& $10^{-8}$  & 12 & 7  \\
	& $10^{-10}$ & 16 & 8  \\
	\midrule
	\multirow{3}{*}{SiH4} 
	& $10^{-6}$  & 20 & 10 \\
	& $10^{-8}$  & 25 & 13 \\
	& $10^{-10}$ & -- & 17 \\
	\bottomrule
	\end{tabular}
	\caption{Number of iterations required by LOBP4dCG and BOSP with $n_{e}=n_b=10$ 
	in \textbf{Example} \ref{exam:small_nev} for different tolerances.}
	\label{tab:iter_comparison}
	\end{table}

	Then, We present how the normalized residuals descend to target tolerances
during the iteration in Figure \ref{fig:res_Na2_SiH4}, from which we can observe a rapid and stable decrease in the normalized residuals for 
different precision requirements. Then, to further investigate the normalized residuals reduction process,
we apply a linear regression analysis for the first ten smallest positive eigenvalues and the associated eigenvectors.
In particular, we denote $r_{\ell}^{(j)}$ as the normalized residuals \eqref{equ:tol}
of $\ell$-th eigenpair $\{\lambda_{\ell}^{(j)}; \xi_{\ell}^{(j)}\}$ in the $j$-th iteration,
%\begin{equation*}
%r_{\ell}^{(j)}
%:= 
%\frac{\Vert H\xi_{\ell}^{(j)} - \lambda_{\ell}^{(j)} \xi_{\ell}^{(j)} \Vert_2}{(1+\lambda_{\ell}^{(j)}) \Vert  \xi_{\ell}^{(j)} \Vert_2},
%\end{equation*}
 and assume $r_{\ell}^{(j)}$ satisfy the following relation
\begin{equation}\label{equ:regression}
	r_{\ell}^{(j)} = \alpha_{\ell}\cdot (r_{\ell}^{(j-1)} )^{\beta_{\ell}}.
\end{equation}
We compute the regression coefficients $\alpha_{\ell}$ and $\beta_{\ell}$ for $\ell =1,2,\ldots,10$, and 
present the first ten smallest positive eigenvalues 
and the regression coefficients for Na2 and SiH4 problems in Table \ref{tab:rate_res_na2_sih4},
from which we observe a superlinear convergence for both problems.

\begin{figure}[htbp]
	\centering
	\subfigure[Na2 with $\varepsilon_{\tt tol} = 10^{-6}$]{
	\includegraphics[scale=0.23]{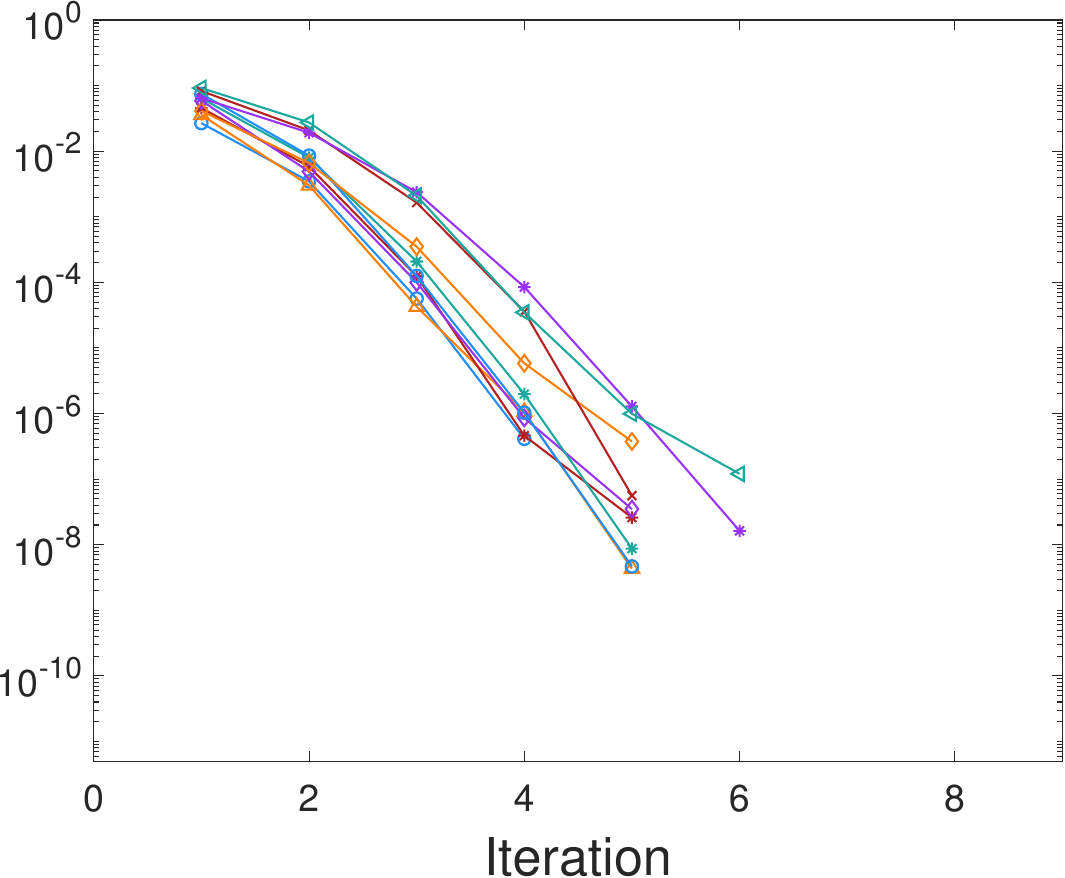}}
	\subfigure[Na2 with $\varepsilon_{\tt tol} = 10^{-8}$]{
	\includegraphics[scale=0.23]{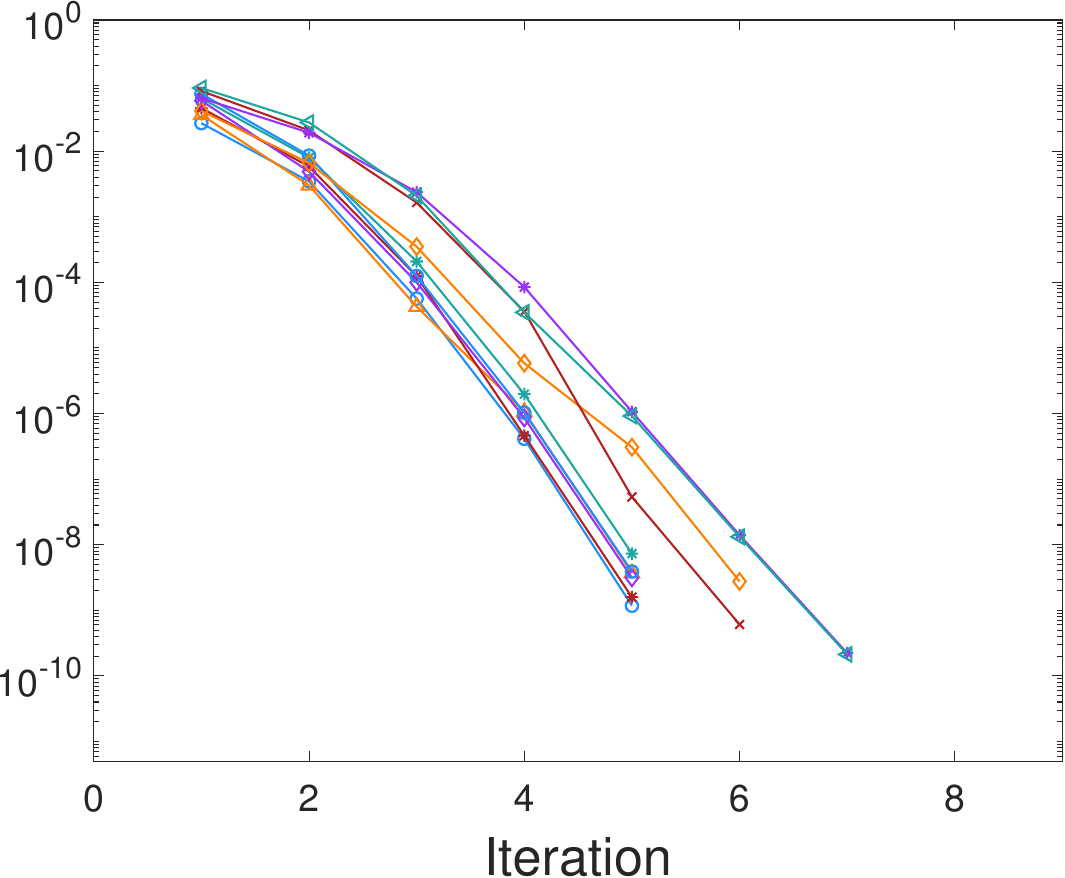}}
	\subfigure[Na2 with $\varepsilon_{\tt tol} = 10^{-10}$]{
	\includegraphics[scale=0.23]{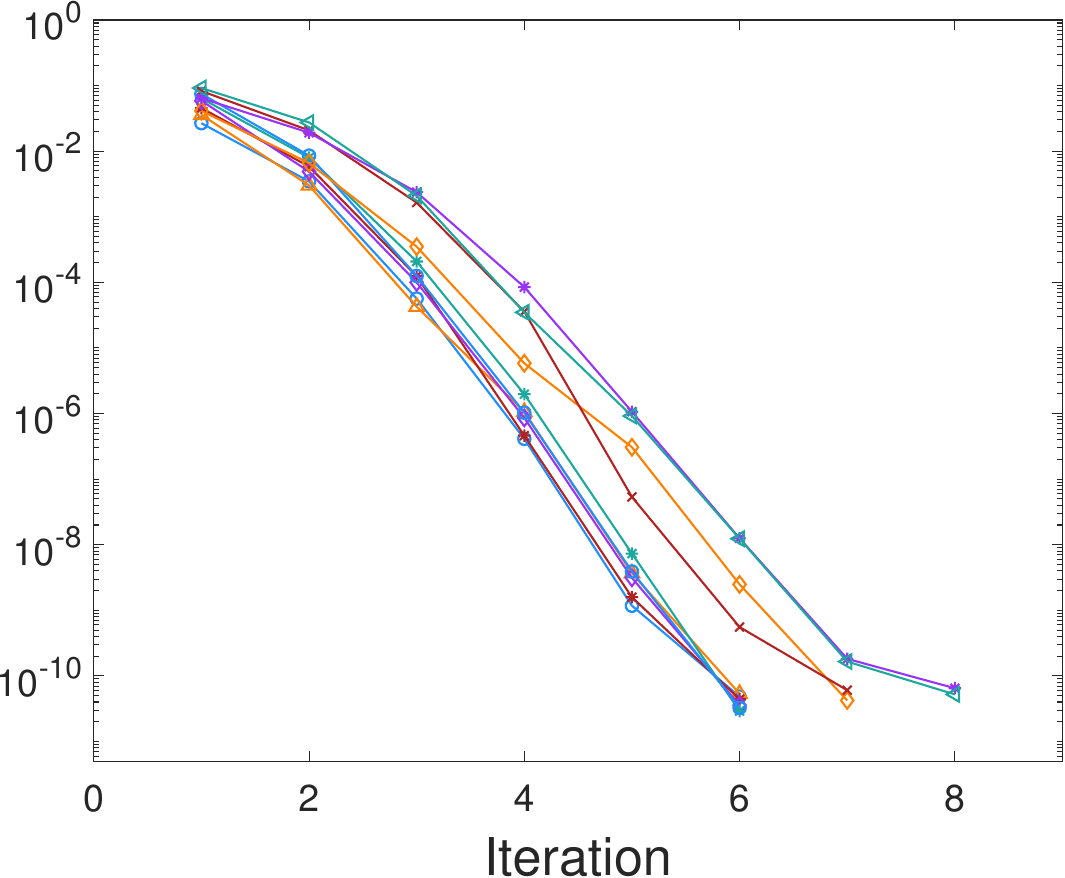}}\\
	\subfigure[SiH4 with $\varepsilon_{\tt tol} = 10^{-6}$]{
	\includegraphics[scale=0.23]{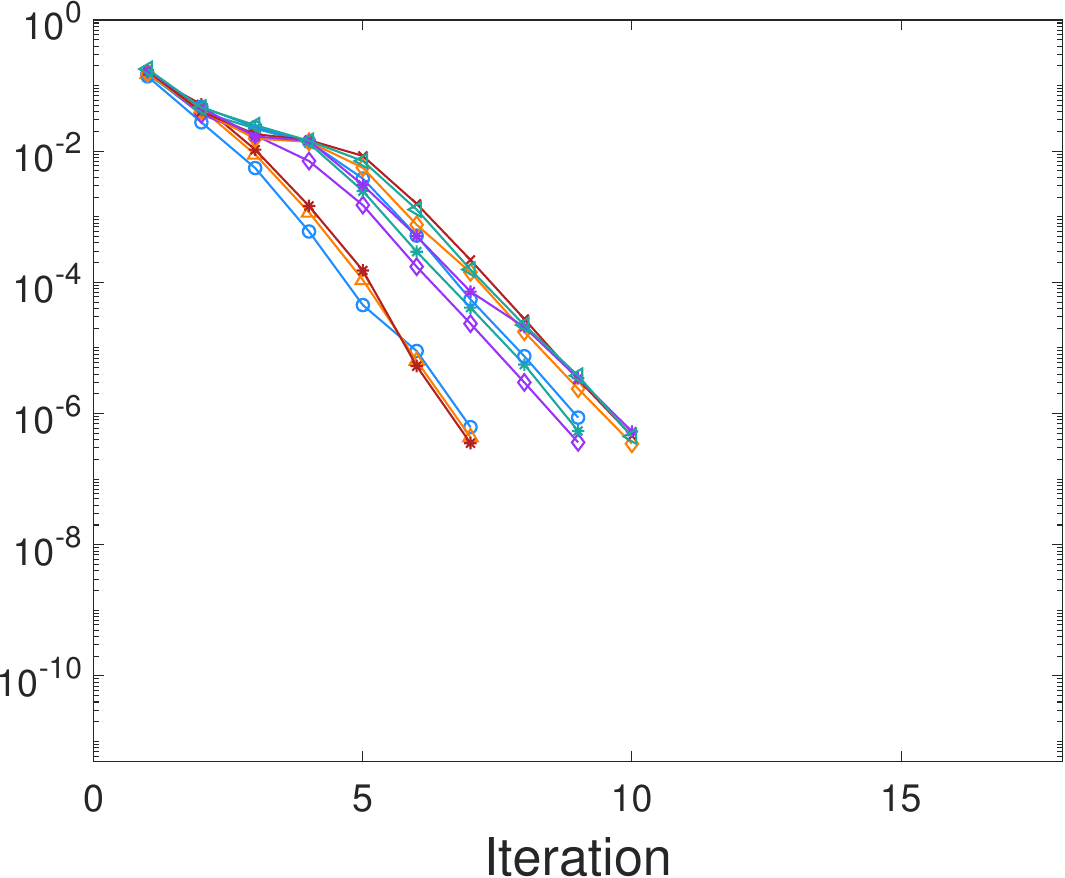}}
	\subfigure[SiH4 with $\varepsilon_{\tt tol} = 10^{-8}$]{
	\includegraphics[scale=0.23]{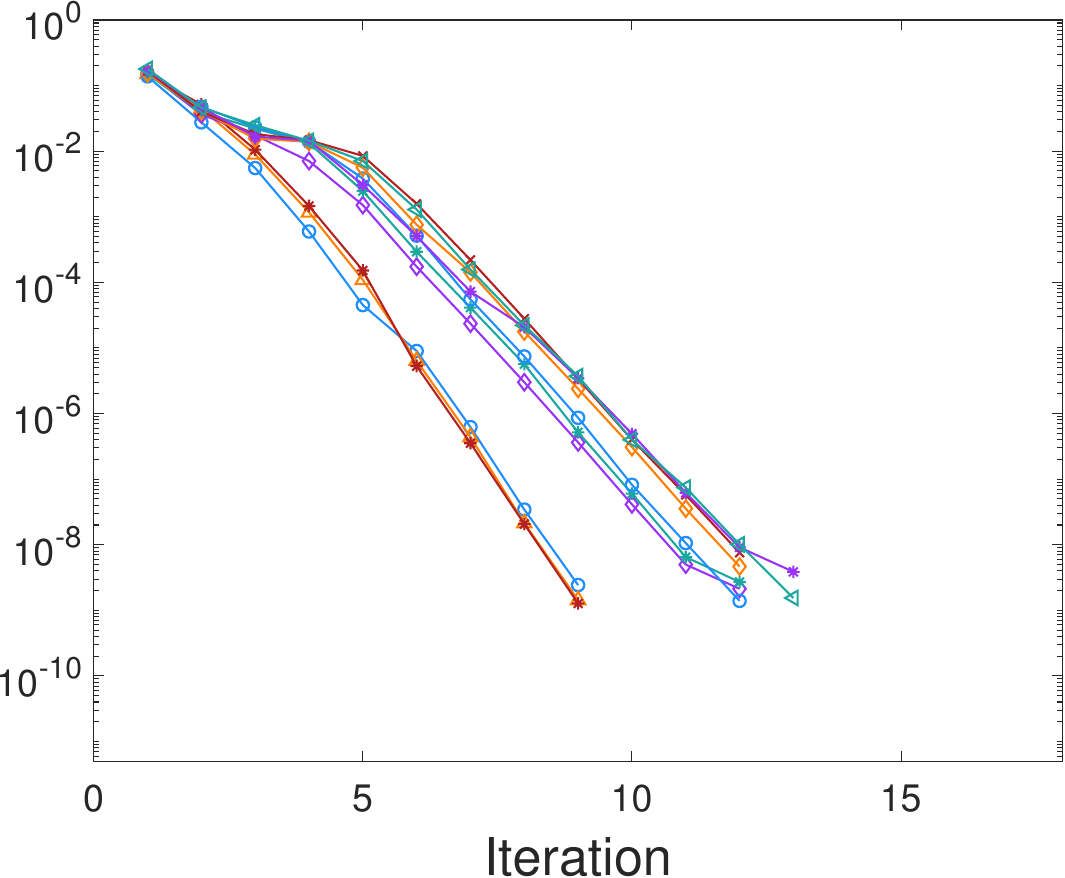}}
	\subfigure[SiH4 with $\varepsilon_{\tt tol} = 10^{-10}$]{
	\includegraphics[scale=0.23]{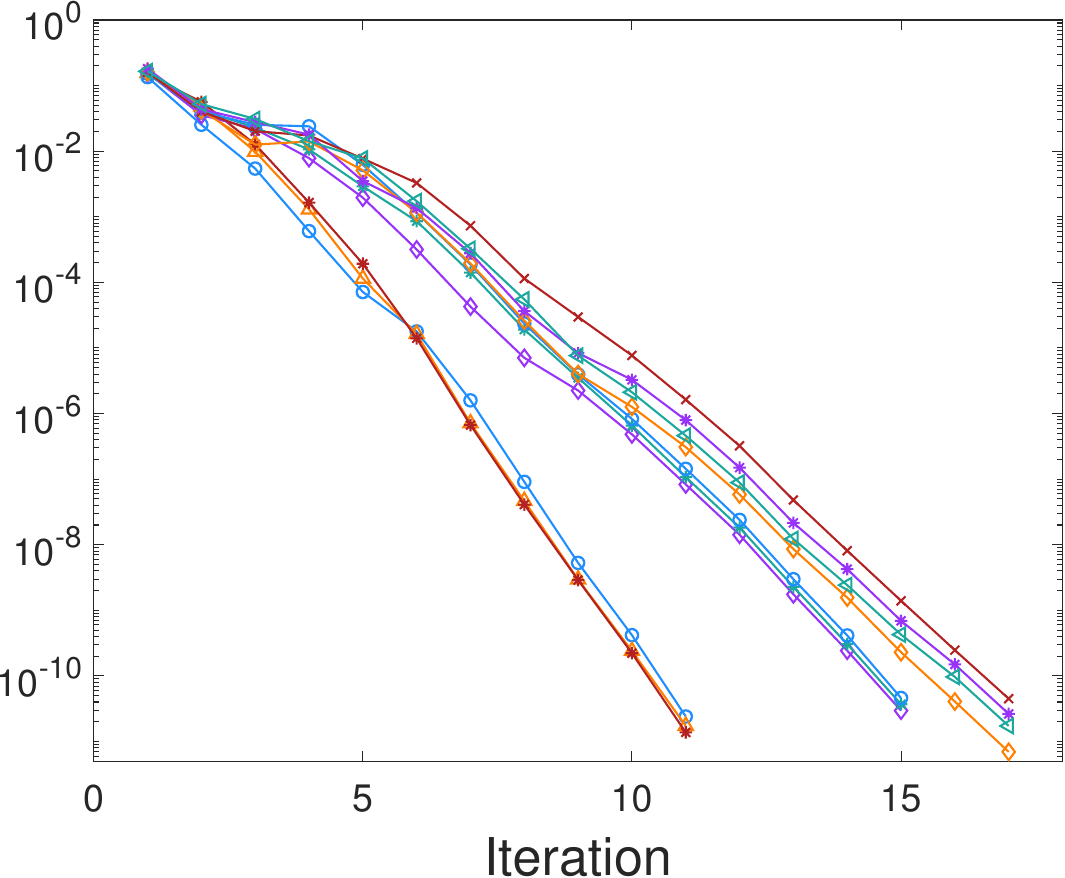}}
	\caption{Descending trend of the normalized residuals with $n_{e}=n_b=10$ in \textbf{Example} \ref{exam:small_nev}.}
	\label{fig:res_Na2_SiH4}
\end{figure}

%\begin{figure}[htbp]
%	\centering
%	\subfigure[Na2 with $\varepsilon_{\tt tol} = 10^{-6}$]{
%	\includegraphics[scale=0.30]{../imag/Na2e_6_res_new.eps}}\qquad
%	\subfigure[SiH4 with $\varepsilon_{\tt tol} = 10^{-6}$]{
%	\includegraphics[scale=0.29]{../imag/SiH4e_6_res_new.eps}}\\
%	\subfigure[Na2 with $\varepsilon_{\tt tol} = 10^{-8}$]{
%	\includegraphics[scale=0.30]{../imag/Na2e_8_res_new.eps}}\qquad
%	\subfigure[SiH4 with $\varepsilon_{\tt tol} = 10^{-8}$]{
%	\includegraphics[scale=0.29]{../imag/SiH4e_8_res_new.eps}}\\
%	\subfigure[Na2 with $\varepsilon_{\tt tol} = 10^{-10}$]{
%	\includegraphics[scale=0.30]{../imag/Na2e_10_res_new.eps}}\qquad
%	\subfigure[SiH4 with $\varepsilon_{\tt tol} = 10^{-10}$]{
%	\includegraphics[scale=0.29]{../imag/SiH4e_10_res_new.eps}}
%	\caption{Descending trend of the normalized residuals  with ${\tt nev}=n_b=10$ in \textbf{Example} \ref{exam:small_nev}.}
%	\label{fig:res_Na2_SiH4}
%\end{figure}

\begin{table}[!htbp]
	\def\temptablewidth{1\textwidth}\tabcolsep 10pt
	\centering
	{\rule{\temptablewidth}{1pt}}
	% \begin{tabularx}{\temptablewidth}{@{\extracolsep{\fill}}r|ccc|ccc}
	\resizebox{\linewidth}{!}{
		% \begin{tabularx}{\temptablewidth}{@{\extracolsep{\fill}}c|ccc|ccc}
			\begin{tabular}{c|c|cc|c|cc}
		&  \multicolumn{3}{c|}{Na2} & \multicolumn{3}{c}{SiH4} \\
		\hline\rule{0pt}{10pt}
    ${\ell}$ & $\lambda_{\ell}$ & $\alpha_{\ell}$ & $\beta_{\ell}$ & $\lambda_{\ell}$ & $\alpha_{\ell}$ & $\beta_{\ell}$\\
		\hline\rule{0pt}{10pt}
		\,1& 0.227852000655427 & 0.0352 & 1.0625 & 0.541531082121720 & 0.2158 & 1.0653 \\
		~2& 0.240179032291449 & 0.0520 & 1.1057 & 0.541531084058815 & 0.1950 & 1.0603 \\
		~3& 0.240179035041609 & 0.0359 & 1.0790 & 0.541531086411073 & 0.2098 & 1.0688 \\
		~4& 0.242493850512922 & 0.0480 & 1.1185 & 0.619960061391761 & 0.3358 & 1.0461 \\
		~5& 0.245728485004234 & 0.0938 & 1.2091 & 0.619960062445083 & 0.3816 & 1.0583 \\
		~6& 0.246402942128366 & 0.0512 & 1.1354 & 0.619960064861915 & 0.3986 & 1.0626 \\
		~7& 0.301737363116933 & 0.1137 & 1.1214 & 0.635136237062253 & 0.3486 & 1.0374 \\
		~8& 0.301737364301465 & 0.0500 & 1.0501 & 0.635136238658330 & 0.4373 & 1.0520 \\
		~9& 0.329112160655168 & 0.0665 & 1.0228 & 0.639017048467930 & 0.3666 & 1.0371 \\
		10& 0.336776988704641 & 0.0557 & 1.0139 & 0.639017049742115 & 0.3667 & 1.0385 \\
		% & average $\beta_k$ && 1.1771 && 1.0772 && 1.0663 \\
	% \end{tabularx}
\end{tabular}
}
	{\rule{\temptablewidth}{1pt}}
	\caption{The regression coefficients on the normalized residuals in \textbf{Equation} \eqref{equ:regression} 
	with $n_{e}=n_b=10$ and $\varepsilon_{\tt tol}=10^{-10}$.}
	\label{tab:rate_res_na2_sih4}
\end{table}

%\
%\vspace{0.2cm}
\begin{remark}[Newton search direction]
Numerical results shown in \textbf{Example} \ref{exam:small_nev}
%of Section \ref{sec:numerical} 
witness a superlinear convergence, 
and it can be interpreted intuitively from the viewpoint of Newton search to find zeros of a vector-valued function.
In fact, it is easy to check that solution 
$\begin{bsmallmatrix} z_i \\ w_i \end{bsmallmatrix}$ 
to the linear system \eqref{equ:residual_ZW}
is close to the Newton search direction at point 
$\begin{bsmallmatrix} \tilde{y}_i \\ \tilde{x}_i \end{bsmallmatrix}$ 
for the following optimization problem
	\begin{equation*}
		\label{equ:uncon_trace}
		\min_{x,y\in\mathbb{R}^n}
		\frac{1}{2}
		(x^{\top}Kx + y^{\top}My)
		-\hat{\lambda}_i(x^{\top}y-1),
	\end{equation*}
which is usually used to solve the constrained problem
\begin{equation*} \label{equ:con_trace}
\min_{x^{\top}y=1} \frac{1}{2} (x^{\top}Kx + y^{\top}My) 
\end{equation*} by augmented Lagrangian method \cite{Hestenes1969Multiplier,Powell1969method}.
\end{remark}

\vspace{0.2cm}
\begin{example}[Large $n_{e}$ case]\label{exam:large_nev}
Here we investigate the efficiency improvement of moving mechanism when computing large number of eigenpairs, that is, 
$n_{e}  \gg 1$.  
In this case, we compute the first $5000$ smallest positive eigenvalues for SiH4 problem, whose size is $n= 5660$. 
The algorithm is implemented with $n_b =150$ and $\varepsilon_{\tt tol} = 10^{-6}$ if not stated otherwise.
\end{example} 

Firstly, we compare efficiency performance with/without the moving mechanism 
for different $n_{e}$, and present all the computation time 
%in Table \ref{tab:nev1-5000sih4}, 
in Figure \ref{fig:nev1-5000sih4}, 
from which we can observe an ascending efficiency improvement for increasing number of required eigenpairs. 
The computation time with the moving mechanism is only \revise{one-sixth} of that without the mechanism for $n_{e}=5000$ (nearly 90\% of all eigenvalues).
Meanwhile, we present the computational time for each component 
in Table \ref{tab:nev5000sih4},
%for ${\tt nev} = 5000$ with $n_b=150$ and $\varepsilon_{\tt tol}=10^{-6}$.
from which one can conclude that the moving mechanism greatly improves efficiency, 
since the size of small-scale LREP \eqref{equ:small_eigen}, is much smaller.

%\begin{table}[!htbp]
%	\centering
%	\setlength{\tabcolsep}{5mm}{
%\begin{tabular}{cccc}
%\toprule
%${\tt nev}$  & \mbox{Moving}   & \mbox{Without moving}  & Ratio     \\
%  \midrule
%$1000$       & ~67.50   & ~131.39  &  51.37\%  \\   % 1.66  
%$2000$       & 136.13   & ~395.97  &  34.38\%  \\   % 2.31  
%$3000$       & 223.89   & ~792.58  &  28.25\%  \\   % 2.79  
%$4000$       & 329.65   & 1704.02  &  19.35\%  \\   % 4.03  
%$5000$       & 431.58   & 3436.40  &  12.57\%  \\   % 6.20         
%\bottomrule
%\end{tabular}
%}
%\caption{The CPU time for different {\tt nev} with/without moving mechanism in \textbf{Example} \ref{exam:large_nev}.}
%\label{tab:nev1-5000sih4}
%\end{table}

\begin{figure}[!htbp]
\centering
\includegraphics[scale=0.5]{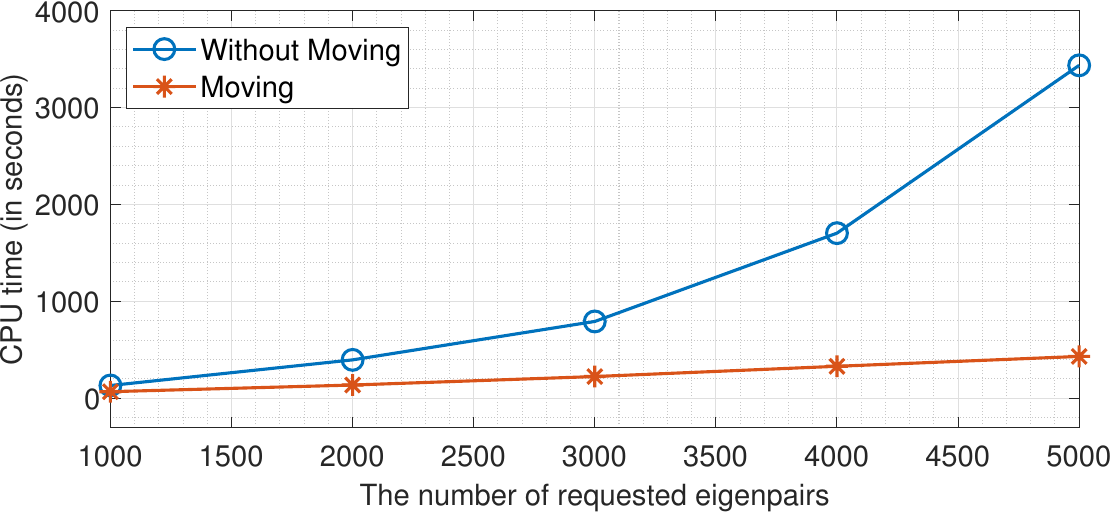}
\caption{The efficiency performance with/without the moving mechanism in \textbf{Example} \ref{exam:large_nev}.}
\label{fig:nev1-5000sih4}
\end{figure}

\begin{table}[!htbp]
	\centering
	\setlength{\tabcolsep}{3pt}{
\begin{tabular}{l|rr|rr}
\toprule
&  \multicolumn{2}{c|}{\mbox{Moving}} & \multicolumn{2}{c}{\mbox{Without moving}} \\
\midrule     
Procedure          & Time  &Percentage & Time & Percentage \\
  \midrule
Initialize $U$ and $V$     & 3.82     & 0.88\%   & 128.53    & 3.74\%        \\
Solve small-scale Problem  & 106.14   & 24.59\%  & 2673.51   & 77.80\%       \\
Check Convergence          & 5.37     & 1.23\%   & 6.07      & 0.18\%        \\
Compute $P$ and $Q$        & 11.47     &2.65\%   & 33.81     & 0.98\%        \\
Compute $W$ and $Z$        & 304.77   &70.61\%  & 594.48    & 17.30\%       \\         
\midrule                                                                  
Total time                 & 431.58   &100.00\%  & 3436.40  & 100.00\%       \\
\bottomrule
\end{tabular}
}
\caption{The computation times with/without the moving mechanism for $n_{e} = 5000$ in \textbf{Example} \ref{exam:large_nev}.}
\label{tab:nev5000sih4}
\end{table}

Secondly, to study the influence of batch size $n_b$ on efficiency, 
we compute the first $5000$ smallest positive eigenvalues and their associated eigenvectors.
In Figure \ref{fig:conpro_sih4_5000_nevConv}, 
we show the number of converged eigenvalues, ${\tt nevConv}$, 
versus the iteration number and CPU time for different batch sizes $n_b= 60,90,120,150,180$.
%In Figure \ref{fig:conpro_sih4_5000_nevConv} (b), we display the iterative number and CPU time needed for the first 
%${\tt nevConv}=1000, 2000, 3000, 4000, 5000$ eigenpairs to converge for different batch size $n_b$.
It seems that $n_b = 150$ is optimal to achieve the best efficiency in this case. 
Larger batch size leads to less iteration steps, but costs much more time in each small-scale LREP \eqref{equ:small_eigen}, 
and eventually results in a poor efficiency performance.
\begin{figure}[!htbp]
\centering
%\subfigure[The number of converged eigenvalues versus the iteration number for different $n_b$.]
%{
%\includegraphics[scale=0.31]{../imag/nev5000_diffbs_conv_iter.eps}\qquad
%\includegraphics[scale=0.31]{../imag/nev5000_diffbs_conv_time.eps}%\\[1.0em]
%}
%\subfigure[The iteration number and CPU time when the first {\tt nevConv} eigenpairs converged.]
{
\includegraphics[scale=0.32]{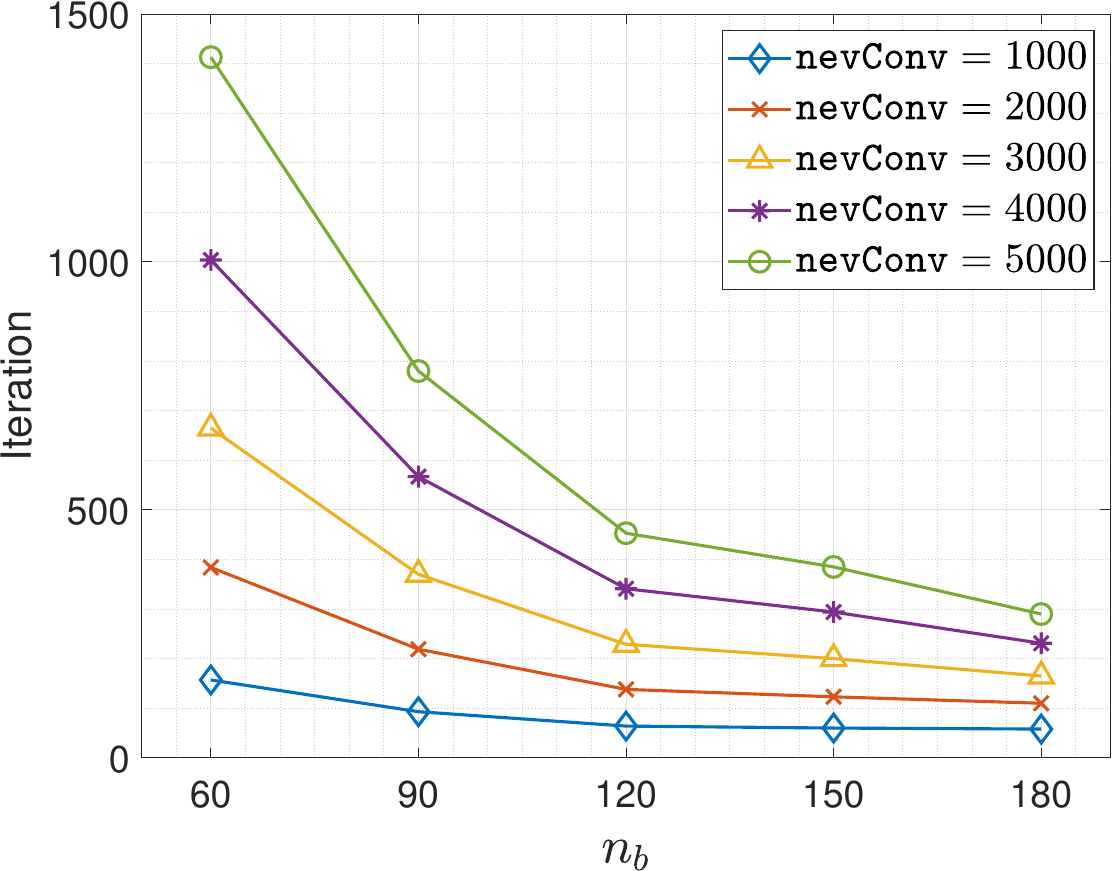}\qquad
\includegraphics[scale=0.32]{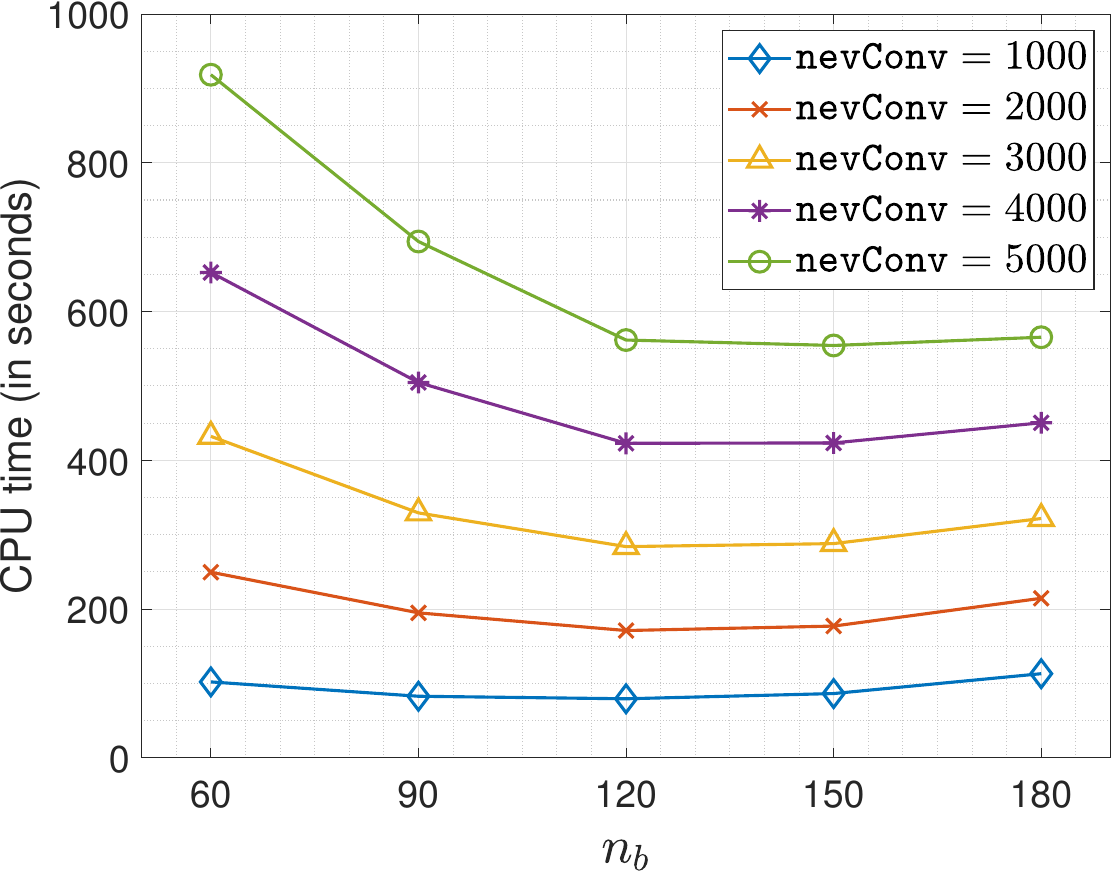}
}
\caption{
	The convergence performance when 
	the first {\tt nevConv} eigenpairs converged
%	$n_{e}=5000$ 
	in \textbf{Example} \ref{exam:large_nev}.
%The number of converged eigenvalues versus the iteration number for different $n_b$ (top left to right),
%and
%the iteration number and CPU time when the first {\tt nevConv} eigenpairs converged (bottom left to right).
}
\label{fig:conpro_sih4_5000_nevConv}
\end{figure}

%========================
% SCALABILITY TESTS
%========================
\subsection{Parallel scalability for large-scale problem}\label{sec:parallel}

In this section, we investigate the parallel scalability for large-scale sparse and dense matrices.

%\vspace{0.2cm}
\begin{example}[Large-scale sparse matrix]\label{exam:large_fem}
In this example, we solve a large-scale sparse LREP, that is generated from finite element discretization of the following eigenvalue problem:
To find 
$\left\{\lambda;u,v\right\}$, 
such that
\begin{eqnarray}\label{equ:Laplace_Eigenvalue_Problem}
\left\{
\begin{array}{ll}
	\begin{bmatrix}
	\revise{O}&-\Delta \\ \mathcal{I} & \revise{O}
	\end{bmatrix}
	\begin{bmatrix}
		v \\ u
	\end{bmatrix}
	= \lambda
	\begin{bmatrix}
		v \\ u
             \end{bmatrix}, &\mbox{ in}\ \Omega,\\[1em]
v=u=0, &\mbox{ on}\ \partial\Omega,
\end{array}
\right.
\end{eqnarray}
where $\Omega=(0,1)^3$.
Here,
$\Delta$ and $\mathcal{I}$
are the Laplacian operator and the identity operator respectively.
%In fact, LREP \eqref{equ:Laplace_Eigenvalue_Problem}
%is equivalent to Laplace eigenvalue problem 
%\begin{eqnarray*}
%\left\{
%\begin{array}{ll}
   %-\Delta u(\bx) = \mu ~u(\bx), &~\forall~~ \bx \in  \Omega,\\[0.5em]
%u(\bx)=0, &~\forall~ ~\bx \in \partial\Omega,
%\end{array}
%\right.
%\end{eqnarray*}
%with $v=u$ and $\lambda=\sqrt{\mu}$.

The discretization of equation \eqref{equ:Laplace_Eigenvalue_Problem} by cubic finite element (P3 element) using PHG\footnote{http://lsec.cc.ac.cn/phg/index\_en.htm}
with $393, 216$ elements results in the following generalized eigenvalue problem
\begin{equation}\label{equ:gLREP}
\begin{bmatrix}	O & K \\ M & O \end{bmatrix}
\begin{bmatrix} y \\ x \end{bmatrix} 
=\lambda 
\begin{bmatrix} B & O \\ O & B \end{bmatrix}
\begin{bmatrix} y \\ x \end{bmatrix},
\end{equation}
where the stiffness and mass matrix $K,M$ are of the same size, i.e., $n  = 1,742,111$ (about 1.7 million).

\end{example}

The numerical experiments were carried out on 
LSSC-IV\footnote{http://lsec.cc.ac.cn/chinese/lsec/LSSC-IVintroduction.pdf} 
in the State Key Laboratory of Scientific and Engineering Computing, Chinese Academy of Sciences. 
Each computing node has two 18-core Intel Xeon Gold 6140 processors at 2.3 GHz and 192 GB memory.
The algorithm is implemented with MPI parallelism in order to measure its scaling performance, using $36$, $72$, $144$, $288$, $576$ processes. 
Table \ref{tab:nev50_linres} shows the computation time of each component step for different number of processes, from which 
one can easily observe a nearly linear scaling with number of processes.

\begin{table}[!htbp]
\centering
\resizebox{\linewidth}{!}{
\begin{tabular}{l|rr|rr|rr|rr|rr}
\toprule
{Number of processes}          & {36}&                      & {72} &                   & {144} &                & {288} &                & {576} & \\
 \midrule
Procedure 
			&Time    & Percentage & Time    & Percentage 
			&Time    & Percentage & Time    & Percentage
			&Time    & Percentage \\
		  \midrule
Initialize $U$ and $V$            & 18.58     &  8.88\%  &  9.74   &  8.61\%  & 4.95 &  8.44\%   & 2.60&  8.52\%  & 1.62&   9.29\%  \\
Solve small-scale Problem & 18.46   &  8.83\%  &  9.90    &  8.75\%   & 5.41 &  9.23\%   & 2.86&  9.38\%  & 1.87&   10.72\%  \\
Check Convergence& 3.77      &  1.80\%  &  2.09   &  1.85\%   & 1.29 &  2.20\%   & 0.67&  2.20\%   & 0.41&   2.35\%  \\
Calculate $P$ and $Q$          & 0.53      &  0.25\%  &  0.30   &  0.47\%   & 0.16 &   0.48\%   & 0.07&  0.23\%  & 0.04&   0.23\%  \\
Calculate $W$ and $Z$       &167.80   & 80.23\% &91.15   & 80.54\%    & 46.84&  79.88\% &24.30& 79.67\% &13.48&  77.29\%  \\
 \midrule
Total Time                 & 209.15 &100.00\% &113.18 &100.00\% &58.64 &100.00\% &30.50 &100.00\% &17.44 & 100.00\%  \\
 %\midrule
%{Number of iterations}& {21} & & {24} & & {24}& & {24} & & {24} &\\
 %{Speedup}& {--} & & {0.92} & & {0.89}& & {0.86} & & {0.75} &\\
                        \bottomrule
\end{tabular}
}
\caption{The parallel scalability in \textbf{Example} \ref{exam:large_fem} with $n_{e}=50$ and $n_b=10$.}
	\label{tab:nev50_linres}
\end{table}

%====================BDG equations================
\begin{example}[Large-scale dense matrix]\label{exam:large_dense}
   In this example, we solve BdG \cite{Gao2020Numerical,Tang2022spectrally}, 
a special LREP that arises from 
the perturbation of the Bose-Einstein Condensates around stationary states. 
The differential operator is spatially discretized by the Fourier spectral method and 
the corresponding discrete system is of \textbf{large-scale}, especially for two and three spatial dimensions, 
and \textbf{dense}, i.e., fully populated for each matrix entry. Therefore, we design a \textbf{matrix-free} interface, 
which requires the matrix-vector product during the iterative process, 
to reduce the memory requirement substantially.
\end{example}

Here, we consider the BdG equation in $3$-dimensional space which is given as follows
\cite{Bao2013Mathematical,Gao2020Numerical,Lucero2008Molecular,Tang2022spectrally}:
to find 
$\left\{\lambda;u,v\right\}$ 
such that
\begin{equation} \label{equ:bdg}
	\begin{cases}
		\mathcal{L}_{\mbox{\tiny GP}}u + \beta|\phi_{\rm s}|^2 u
		+\beta\phi_{\rm s}^2 v
        = \lambda ~u,\\[0.4em]
				\mathcal{L}_{\mbox{\tiny GP}}v + \beta|\phi_{\rm s}|^2 v
				+\beta\overline{\phi}_{\rm s}^2 u
	= -\lambda~ v,
	\end{cases}
\end{equation}
with constrain 
$\int_{\mathbb{R}^3} \left(|u({\bf x})|^2 -|v({\bf x})|^2\right) \, {\rm d}{\bf x} = 1$,
%Here we take $d=3$ in this example. 
%The linear operator $\mathcal{L}_{\mbox{\tiny GP}}$ reads as follows
where
%\begin{equation*}
$\mathcal{L}_{\mbox{\tiny GP}} =
	-\frac{1}{2}\nabla^2 + V({\bf x}) 
	+ \beta|\phi_{\rm s}|^2 
	- \mu_{\rm s}$
and $V({\bf x})=\frac{1}{2} \left[ (\gamma_xx)^2+(\gamma_yy)^2+(\gamma_zz)^2\right]$ is the harmonic external
potential with $\gamma_x$, $\gamma_y$, $\gamma_z$ being the trapping frequencies in each spatial direction.
The chemical potential $\mu_{\rm s}$ and the stationary state $\phi_{\rm s}({\bf x})$
satisfy the following nonlinear Gross-Pitaevskii equation (GPE) \cite{Bao2013Mathematical} 
\begin{equation} \label{equ:gpe}
	\Big[
		-\frac{1}{2}\nabla^2 + V({\bf x})
		+ \beta|\phi_{\rm s}|^2 
	\Big]\phi_{\rm s} 
	= \mu_{\rm s}\phi_{\rm s},\qquad
	\mbox{  with } 
\int_{\mathbb {R}^3} 
|\phi_{\rm s}({\bf x})|^2\, {\rm d}{\bf x} = 1.
\end{equation}
%Note that the stationary state $\phi_{\rm s}({\bf x})$ is a real-valued function \cite{Bao2013Mathematical}.
Using a change of variables ${f} = u+v$ and ${g} = u-v$, 
we can rewrite the BdG equation \eqref{equ:bdg} as follows
\begin{equation}\label{equ:bdg_KM}
\mathcal{L}_{\mbox{\tiny GP}}{g} = \lambda ~f,\qquad 
(\mathcal{L}_{\mbox{\tiny GP}} + 2\beta|\phi_{\rm s}|^2) {f}
= \lambda~ g,\qquad
	\mbox{with }
	\int_{\mathbb {R}^3} f({\bf x})g({\bf x}) \, {\rm d}{\bf x} = 1.
\end{equation}
The nullspace of $\mathcal{L}_{\mbox{\tiny GP}}$ is not empty since $\mathcal{L}_{\mbox{\tiny GP}}\phi_{\rm s}=0$, 
and it immediately implies that $0$ is an eigenvalue of \eqref{equ:bdg_KM}.
As stated in \cite{Gao2020Numerical,Tang2022spectrally}, there exist three analytical eigenpairs as follows
\begin{equation*}
	\left\{\gamma_{\alpha};
		{\sqrt{2}}\gamma_{\alpha}^{-\frac{1}{2}}\partial_{\alpha}\phi_{\rm s},
		-{\sqrt{2}} \gamma_{\alpha}^{\frac{1}{2}}\alpha\phi_{\rm s}
	\right\},\ \mbox{ for } \alpha = x,y,z,
\end{equation*}
%\begin{equation*}
%	\scalebox{0.85}{$
%	\left\{\gamma_{x};
%		{\sqrt{2}}\gamma_{x}^{-\frac{1}{2}}\partial_{x}\phi_{\rm s},
%		-{\sqrt{2}} \gamma_{x}^{\frac{1}{2}}x\phi_{\rm s}
%	\right\},\
%	\left\{\gamma_{y};
%		{\sqrt{2}}\gamma_{y}^{-\frac{1}{2}}\partial_{y}\phi_{\rm s},
%		-{\sqrt{2}} \gamma_{y}^{\frac{1}{2}}y\phi_{\rm s}
%	\right\},\
%	\left\{\gamma_{z};
%		{\sqrt{2}}\gamma_{z}^{-\frac{1}{2}}\partial_{z}\phi_{\rm s},
%		-{\sqrt{2}} \gamma_{z}^{\frac{1}{2}}z\phi_{\rm s}
%	\right\},$}
%\end{equation*}
and they will serve as benchmark solutions.

Since the eigenfunctions are smooth and fast-decaying, 
we discretize the BdG equation \eqref{equ:bdg_KM} using the Fourier spectral method
on a uniformly discretized rectangular domain $\Omega$.
As a result, we are confronted with a large-scale densely populated eigenvalue problem.
%Explicit matrix storage is quite challenging or even impossible for modern hardware architecture, 
%therefore, it is imperative to design a matrix-free interface, so as to provide the matrix-vector product
%or operation-function evaluation. 
%Accurate and efficient implementation of operator-function computation is 
The operator-function/matrix-vector product implementation is of essential importance to efficiency, and is realized with the Fast Fourier Transform with almost optimal complexity 
$\mathcal O(n\log n)$ with $n$ being the number of grid points.

%As shown in \cite{Gao2020Numerical}, 
%there exists three analytical eigenvalues, i.e., $\gamma_x,\gamma_y$ and $\gamma_z$.
In this case, we set $\gamma_x=\gamma_y=1$, $\gamma_z=2$, $\beta = 100$ and 
compute the first eight smallest positive eigenvalues of \eqref{equ:bdg_KM} with $\varepsilon_{\tt tol} = 10^{-8}$ for different mesh sizes 
on domain $\Omega = (-8,8)^3$. The numerical eigenvalues and computation times are presented in Table \ref{tab:bdg_eval}, from which 
we can observe that the computational time is roughly $\mathcal{O}(n\log n)$ and 
the number of matrix-vector product seems unchanged and independent of the spatial mesh size $h$.
It implies that the performance remains efficient and scalable for large-scale matrix.
 \begin{table}[!htbp]
	\centering
\resizebox{\linewidth}{!}{
	% {\rule{\temptablewidth}{1pt}}
	% \begin{tabularx}{\temptablewidth}{@{\extracolsep{\fill}}c|ccccc}
		\begin{tabular}{c|ccccc}
			\toprule
$h$ &	$h_0 =1$ & $h_0/2$ & $h_0/4$ & $h_0/8$ & $h_0/16$ \\
		\hline\rule{0pt}{10pt}
$n$	&	 4,096  &    32,768  &   262,144  &  2,097,152 &  16,777,216 \\
		\hline\rule{0pt}{10pt}
$\lambda_1$  &1.00482127487320 & 0.99999662331167 & 1.00000000000871 & 0.99999999999422 & 1.00000000000157 \\
$\lambda_2$  &1.00482127490807 & 0.99999662331498 & 1.00000000000872 & 0.99999999999422 & 1.00000000000215 \\
$\lambda_3$  &1.51719126293581 & 1.51548541163141 & 1.51549662106678 & 1.51549662107412 & 1.51549662106231 \\
$\lambda_4$  &1.54237926065405 & 1.51549063915454 & 1.51549662106947 & 1.51549662107433 & 1.51549662106233 \\
$\lambda_5$  &1.90225282934626 & 1.87508846000935 & 1.87496258699030 & 1.87496258650835 & 1.87496258643911 \\
$\lambda_6$  &1.91639375724723 & 2.00026954487640 & 2.00000000104929 & 2.00000000005089 & 2.00000000000583 \\
$\lambda_7$  &2.06343832504500 & 2.04187426614331 & 2.04187975275431 & 2.04187975276325 & 2.04187975274606 \\
$\lambda_8$  &2.06343832505536 & 2.04187426614365 & 2.04187975275628 & 2.04187975276327 & 2.04187975274609 \\
		\hline\rule{0pt}{10pt}
		 CPU(s) & 0.66  & 4.39 &  37.34 & 397.02 & 3706.27 \\
		 \bottomrule
	\end{tabular}
	% {\rule{\temptablewidth}{1pt}}
}
        \caption{The first eight smallest positive eigenvalues and CPU time
           in \textbf{Example} \ref{exam:large_dense}.}
	\label{tab:bdg_eval}
\end{table}

To further investigate the performance, in Figure \ref{fig:bdg_err}, 
we present the normalized residuals of $\gamma_x$ and $\gamma_z$ (top left and right) 
and eigenvalue errors of $\gamma_x$ and $\gamma_z$ (bottom left and right) 
versus 
iteration steps for different mesh sizes $h$. 
It is noticed that the convergence seems to be independent 
of the mesh size, while the bottom row indicates a similar spectral convergence as its analytical counterparts.  
A more comprehensive study of its performance, in terms of accuracy, efficiency, and parallel capability, in the context 
of BdG equation is going on, and shall be reported in another article.
\begin{figure}[!htbp]
    \centering
 		\subfigure[The normalized residuals of $\gamma_x$ and $\gamma_z$.]
		{
     \includegraphics[scale=0.33]{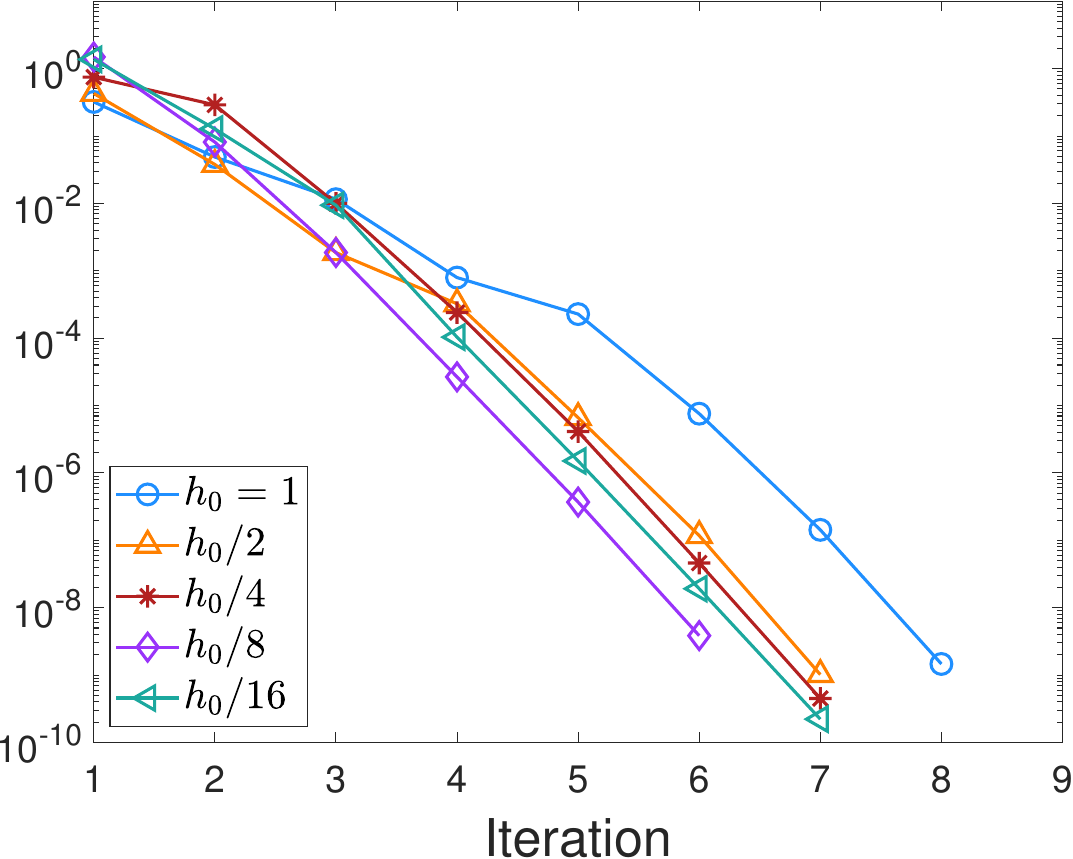}\qquad
     \includegraphics[scale=0.33]{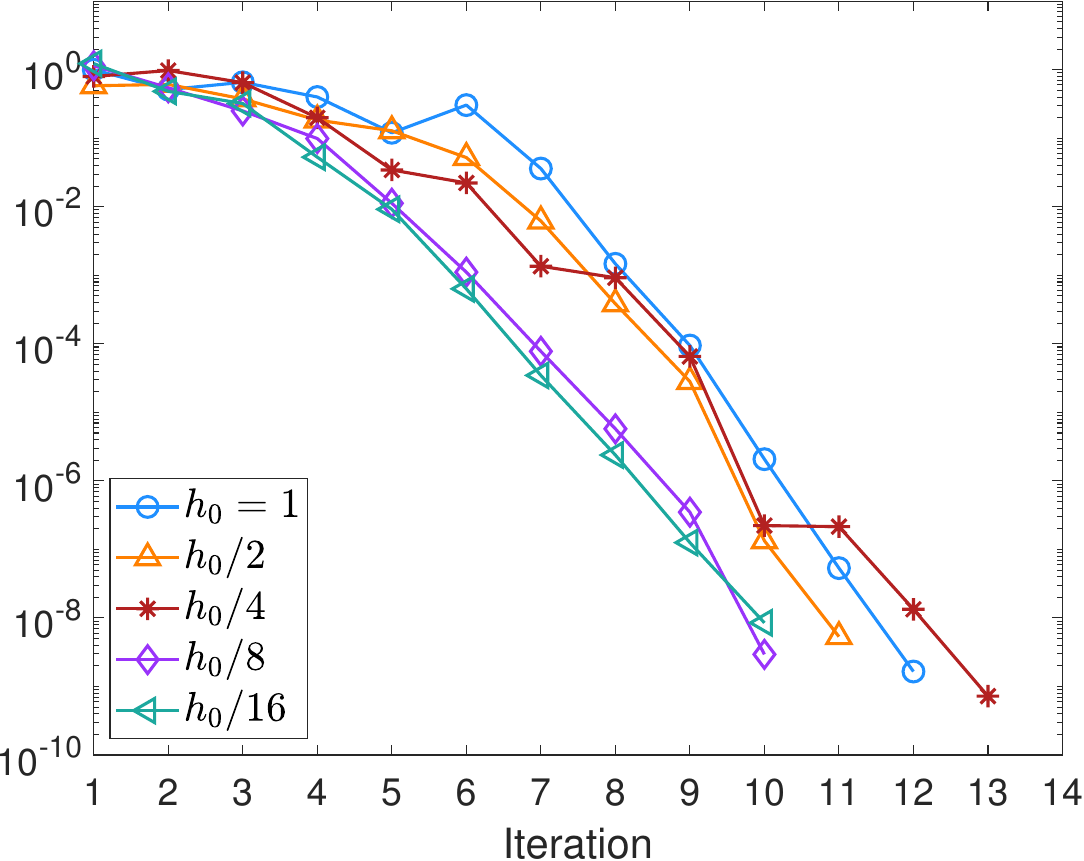}% \\[1.0em]
		}
 		\subfigure[The numerical errors of $\gamma_x$ and $\gamma_z$.]
 		{
 		\includegraphics[scale=0.33]{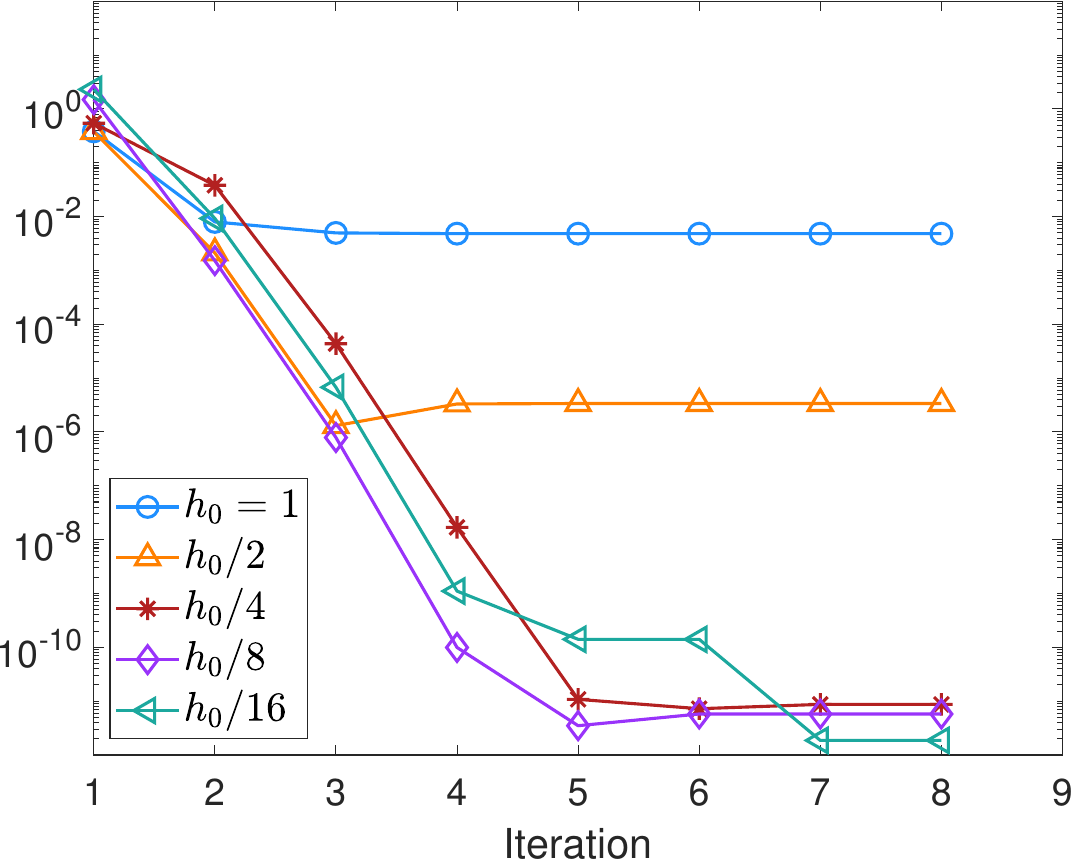}\qquad
 		\includegraphics[scale=0.33]{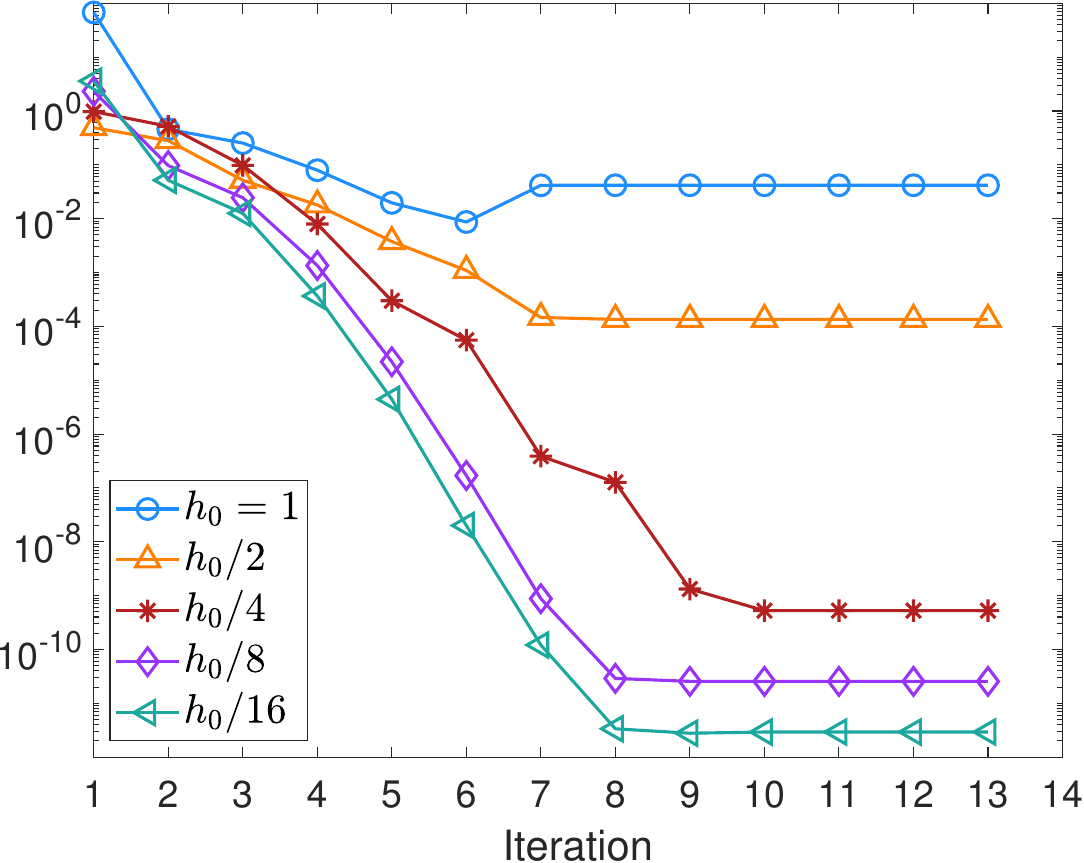}
 		}
     \caption{The convergence performance %for $\gamma_x$ and $\gamma_z$
			 in \textbf{Example} \ref{exam:large_dense} for different mesh sizes.
%			 The normalized residuals of $\gamma_x$ and $\gamma_z$ (top left to right), 
%    and numerical errors of $\gamma_x$ and $\gamma_z$ (bottom left to right).
		 }
   \label{fig:bdg_err}
\end{figure}
%%%===============================concluding================================
\section{Conclusions}

For linear response eigenvalue problems,
we proposed a Bi-Orthogonal Structure-Preserving subspace iterative solver
to compute the corresponding large-scale sparse and dense discrete eigenvalue system that may admit zero eigenvalues.
Our solver is based on a direct sum decomposition of biorthogonal invariant subspaces where the generalized nullspace is taken into account.
Preserving the biorthogonal property is of essential importance to guarantee the convergence, efficiency, and parallel scalability.
The MGS-Biorth algorithm %, 
%together with a block version to improve parallel scalability, 
preserves the biorthogonality quite well on the discrete level.
We introduce a deflation mechanism via biorthogonalization without introducing any artificial parameters so that 
the converged eigenpairs are deflated out automatically. 
When the number of requested eigenpairs is very large, we propose a moving mechanism to advance the computation batch by batch such that
the memory requirements are greatly reduced and efficiency is much improved to a great extent.
For large-scale problems, the matrix-vector product is allowed to be provided interactively, waiving explicit memory storage.
The performance is further improved when the matrix-vector product is implemented using parallel computing.
Extensive numerical results are shown to demonstrate the stability, efficiency, and parallel scalability.

The linear response eigenvalue problem in this article is equivalent to the 
real Bethe-Salpeter eigenvalue problem. In contrast, complex Bethe-Salpeter eigenvalue problem, which is quite common in many physical applications,  
is much more challenging and now under investigation, 
and we shall report our results in future work.

%\bibliographystyle{plain}

%\bibliographystyle{abbrvnat}
%\bibliography{D:/fullbib/fullbib}

\begin{thebibliography}{38}
      \normalem
\providecommand{\natexlab}[1]{#1}
\providecommand{\url}[1]{\texttt{#1}}
\expandafter\ifx\csname urlstyle\endcsname\relax
\providecommand{\doi}[1]{doi: #1}\else
\providecommand{\doi}{doi: \begingroup \urlstyle{rm}\Url}\fi


\bibitem{Bai2012Minimization}
{\sc Z. Bai and R.C. Li.}
{\em Minimization principles for the linear response eigenvalue problem I: Theory}, 
SIAM J. Matrix Anal. Appl., 33(2012), pp.~1075--1100.

\bibitem{Bai2013Minimization}
{\sc Z. Bai and R.C. Li},
{\em Minimization principles for linear response eigenvalue problem II: Computation}, 
SIAM J. Matrix Anal. Appl., 34(2013), pp.~392--416.

\bibitem{Bai2014Minimization}
{\sc Z. Bai and R.C. Li},
{\em Minimization principles and computation for the generalized linear response eigenvalue problem}, 
BIT Numer. Math., 54(2014), pp.~31--54.

\bibitem{Bai2016Linear}
{\sc Z. Bai, R.C. Li and W.W. Lin},
{\em Linear response eigenvalue problem solved by extended locally optimal preconditioned conjugate gradient methods}, 
Sci. China Math., 59(2016), pp.~1443--1460.

\bibitem{Bao2013Mathematical}
{\sc W. Bao and Y. Cai},
{\em Mathematical theory and numerical methods for Bose-Einstein condensation},
Kinet. Relat. Mod., 6(2013), pp.~1--135.

%\bibitem{Bao2015Dimension}
%{\sc W. Bao, L. Le Treust and F. M{\'e}hats}, 
%{\em Dimension reduction for anisotropic Bose-Einstein condensates in
%the strong interaction regime}, 
%Nonlinearity, {\bf 28} (3) (2015), 755.

\bibitem{Benner2022efficientbse}
{\sc P. Benner and C. Penke},
{\em Efficient and accurate algorithms for solving the Bethe-Salpeter eigenvalue problem for crystalline systems},
J. Comput. Appl. Math., 400(2022), pp.~0377--0427.

\bibitem{Blase2020Bethe}
{\sc X. Blase, I. Duchemin, D. Jacquemin and P. F. Loos}, 
{\em The Bethe-Salpeter equation formalism: from physics to chemistry}, 
J. Phys. Chem. Lett., 11(2020), pp.~7371--7382.

\bibitem{Brabec2015Efficient}
{\sc J. Brabec, L. Lin, M. Shao, N. Govind, C. Yang, Y. Saad and E. G. Ng}, 
{\em Efficient algorithms for estimating the absorption spectrum within
linear response TDDFT}, 
J. Chem. Theory Comput., 11(2015), pp.~5197--5208.

\bibitem{Casida1995Time}
{\sc M. E. Casida}, 
{\em Time-dependent density functional response theory for molecules in: Recent Advances In Density Functional Methods: (Part I)}, 
World Scientific, 1995, pp.~155--192.

\bibitem{Gao2020Numerical}
{\sc Y. Gao and Y. Cai},
{\em Numerical methods for Bogoliubov-de Gennes excitations of Bose-Einstein condensates},
J. Comput. Phys., 403(2020), article 109058.

\bibitem{Giannozzi2009QUANTUM}
{\sc P. Giannozzi, S. Baroni, N. Bonini, M. Calandra, R. Car, C. Cavazzoni,
D. Ceresoli, G. L. Chiarotti, M. Cococcioni, I. Dabo et al}, 
{\em QUANTUM ESPRESSO: a modular and open-source software project for
quantum simulations of materials}, 
J. Phys. Condens. Matter., 21(2009), article 395502.

\bibitem{Hernandez2005SLEPc}
{\sc V. Hernandez, J. E. Roman, and V. Vidal},
{\em SLEPc: A scalable and flexible toolkit for the solution of eigenvalue problems},
ACM Trans. Math. Softw., 31(2005), pp.~351--362.

%\bibitem{Hestenes1958Inversion}
%{\sc M. R. Hestenes}, 
%{\em Inversion of matrices by biorthogonalization and related results}, 
%J. Soc. Ind. Appl. Math., 6(1958), pp.~51--90.

\bibitem{Hestenes1969Multiplier}
{\sc M. R. Hestenes}, 
{\em Multiplier and gradient methods}, 
J. Optimiz. Theory App., 5(1969), pp.~303--320.


\bibitem{Higham1995test}
{\sc N. J. Higham}, 
{\em The test matrix toolbox for matlab (version 3.0)}, 
Technical report, University of Manchester, 1995.

\bibitem{Kohaupt2014Introduction}
{\sc L. Kohaupt}, 
{\em Introduction to a Gram-Schmidt-type biorthogonalization method}, 
Rocky Mt. J. Math., 44(2014), pp.~1265--1279.

\bibitem{Kress2012Numerical}
{\sc R. Kress}, 
{\em Numerical Analysis}, 
Springer Science \& Business Media, 2012.

\bibitem{Lanczos1950iteration}
{\sc C. Lanczos}, 
{\em An iteration method for the solution of the eigenvalue problem of linear differential and integral operators}, 
J. Res. Nat. Bur. Stand., 45(1950), pp.~255--282.
%Journal of Research of the National Bureau of Standards

\bibitem{Li2023GCGE}
{\sc Y. Li, Z. Wang and H. Xie}, 
{\em GCGE: a package for solving large scale eigenvalue problems by
parallel block damping inverse power method}, 
CCF Trans. High Perform. Comput., 5(2023), pp.~171--190.

\bibitem{Li2020parallel}
{\sc Y. Li, H. Xie, R. Xu, C. You and N. Zhang}, 
{\em A parallel generalized conjugate gradient method for large scale
eigenvalue problems}, 
CCF Trans. High Perform. Comput., 2(2020), pp.~111--122.

\bibitem{Lucero2008Molecular}
{\sc M. J. Lucero, A. M. N. Niklasson, S. Tretiak and M. Challacombe}, 
{\em Molecular-orbital-free algorithm for excited states in time-dependent
perturbation theory}, 
J. Chem. Phys., 129(2008), article 064114.



\bibitem{Martin2021Stability}
{\sc A. D. Martin and P.B. Blakie},
{\em Stability and structure of an anisotropically trapped dipolar Bose-Einstein condensate: Angular and linear rotons},
Phys. Rev. A, 86(2012), article 053623.

\bibitem{Onida2002Electronic}
{\sc G. Onida, L. Reining and A. Rubio}, 
{\em Electronic excitations: density-functional versus many-body
Green's-function approaches}, 
Rev. Mod. Phys., 74(2002), pp.~601--659.

\bibitem{Papakonstantinou2007Reduction}
{\sc P. Papakonstantinou}, 
{\em Reduction of the RPA eigenvalue problem and a generalized
Cholesky decomposition for real-symmetric matrices}, 
Europhys. Lett., 78(2007), article 12001.

\bibitem{Parlett1985look}
{\sc B. N. Parlett, D. R. Taylor and Z. A. Liu}, 
{\em A look-ahead Lanczos algorithm for unsymmetric matrices}, 
Math. Comput., 44(1985), pp.~105--124.

\bibitem{Powell1969method}
{\sc M. J. Powell}, 
{\em A method for nonlinear constraints in minimization problems},
Optimization, Academic Press, 1969, pp.~283--298.

\bibitem{Ring2004Nuclear}
{\sc P. Ring and P. Schuck}, 
{\em The Nuclear Many-body Problem}, 
Springer Science \& Business Media, 2004.

\bibitem{Rocca2012block}
{\sc D. Rocca, Z. Bai, R. C. Li and G. Galli}, 
{\em A block variational procedure for the iterative diagonalization of
non-Hermitian random-phase approximation matrices}, 
J. Chem. Phys., 136(2012), article 034111.

\bibitem{Ruhe1983Numerical}
{\sc A. Ruhe}, 
{\em Numerical aspects of Gram-Schmidt orthogonalization of vectors}, 
Linear Algebra Appl., 52(1983), pp.~591--601.

\bibitem{Saad1982Lanczos}
{\sc Y. Saad},
{\em The Lanczos biorthogonalization algorithm and other oblique
projection methods for solving large unsymmetric systems}, 
SIAM J. Numer. Anal., 19(1982), pp.~485--506.

\bibitem{Saad2003Iterative}
{\sc Y. Saad},
{\em Iterative Methods for Sparse Linear Systems}, 
SIAM, 2003.

\bibitem{Saad2011Numerical}
{\sc Y. Saad},
{\em Numerical Methods for Large Eigenvalue Problems},
J. Soc. Ind. Appl. Math., 2011.

\bibitem{Salpeter1951relativistic}
{\sc E. E. Salpeter and H. A. Bethe}, 
{\em A relativistic equation for bound-state problems}, 
Phys. Rev., 84(1951), pp.~1232-1242.

\bibitem{Shao2018structure}
{\sc M. Shao, F. H. da Jornada, L. Lin, C. Yang, J. Deslippe and S. G. Louie}, 
{\em A structure preserving Lanczos algorithm for computing the optical
absorption spectrum}, 
SIAM J. Matrix Anal. Appl., 39(2018), pp.~683--711.

\bibitem{Shao2016Structure}
{\sc M. Shao, F. H. da Jornada, C. Yang, J. Deslippe and S. G. Louie}, 
{\em Structure preserving parallel algorithms for solving the
Bethe-Salpeter eigenvalue problem}, 
Linear Algebra Appl., 488(2016), pp.~148--167.

\bibitem{Stewart2008Block}
	{\sc G. W. Stewart}, 
{\em Block Gram-Schmidt orthogonalization}, 
SIAM J. Sci. Comput., 31(2008), pp.~761--775.

\bibitem{Strinati1988Application}
{\sc G. Strinati}, 
{\em Application of the Green's functions method to the study of the
optical properties of semiconductors}, 
Riv. del Nuovo Cim., 11(1988), pp.~1--86.

\bibitem{Tang2022spectrally}
{\sc Q. Tang, M. Xie, Y. Zhang and Y. Zhang}, 
{\em A spectrally accurate numerical method for computing the
Bogoliubov-de Gennes excitations of dipolar Bose-Einstein
condensates}, 
SIAM J. Sci. Comput., 44(2022), pp.~B100--B121.

\bibitem{Teng2013Convergence}
{\sc Z. Teng and R.C. Li}, 
{\em Convergence analysis of Lanczos-type methods for the linear
response eigenvalue problem}, 
J. Comput. Appl. Math., 247(2013), pp.~17--33.

\bibitem{Teng2019feast}
{\sc Z. Teng and L. Lu}, 
{\em A FEAST algorithm for the linear response eigenvalue problem}, 
Algorithms, 12(2019), pp.~181.

\bibitem{Teng2017block}
{\sc Z. Teng and L.H. Zhang}, 
{\em A block Lanczos method for the linear response eigenvalue problem}, 
Electron. Trans. Numer. Anal., 46(2017), pp.~505--523.

\bibitem{Teng2016block}
{\sc Z. Teng, Y. Zhou and R.C. Li}, 
{\em A block Chebyshev-Davidson method for linear response eigenvalue
problems}, 
Adv. Comput. Math., 42(2016), pp.~1103--1128.

\bibitem{Thouless1961Vibrational}
{\sc D.J. Thouless}, 
{\em Vibrational states of nuclei in the random phase approximation}, 
Nucl. Phys., 22(1961), pp.~78--95.

\bibitem{Thouless1972Quantum}
{\sc D.J. Thouless}, 
{\em The Quantum Mechanics of Many-Body Systems}, 
Academic, 1972.

\bibitem{Varga2000Matrix}
{\sc R.S Varga}, 
{\em Matrix Iterative Analysis}, 
Springer Science \& Business Media, 2000.

\bibitem{Vecharynski2017Efficient}
{\sc E. Vecharynski, J. Brabec, M. Shao, N. Govind and C. Yang}, 
{\em Efficient block preconditioned eigensolvers for linear response
time-dependent density functional theory}, 
Comput. Phys. Commun., 221(2017), pp.~42--52.

\bibitem{Zhang2020generalized}
{\sc N. Zhang, Y. Li, H. Xie, R. Xu and C. You}, 
{\em A generalized conjugate gradient method for eigenvalue problems}, 
Sci. Sin. Math., 51(2021), pp.~1297--1320.

\end{thebibliography}

\end{document}